\documentclass[11pt,leqno]{article}
\usepackage{amsmath, amscd, amsthm, amssymb, graphics, xypic, mathrsfs, setspace, fancyhdr, times,enumerate}
\entrymodifiers={+!!<0pt,\fontdimen22\textfont2>}
\DeclareFontFamily{OT1}{pzc}{}
\DeclareFontShape{OT1}{pzc}{m}{it}%
             {<-> s * [1.195] pzcmi7t}{}
\DeclareMathAlphabet{\mathscr}{OT1}{pzc}%
                                 {m}{it}
\usepackage[colorlinks=true,pagebackref=true]{hyperref} 
\hypersetup{backref}

\newcommand{\Spec}{\operatorname{Spec}}

\newcommand{\isomto}{{\stackrel{\sim}{\;\longrightarrow\;}}}
\newcommand{\isomt}{{\stackrel{{\scriptscriptstyle{\sim}}}{\;\rightarrow\;}}}

\renewcommand{\O}{{\mathcal O}}
\renewcommand{\hom}{\operatorname{Hom}}

\newcommand{\real}{{\mathbb R}}
\newcommand{\cplx}{{\mathbb C}}

\newcommand{\Z}{{\mathbb Z}}
\newcommand{\N}{{\mathbb N}}
\newcommand{\A}{{\mathbb A}}
\newcommand{\aone}{{\mathbb A}^1}
\newcommand{\pone}{{\mathbb P}^1}

\newcommand{\gm}{{{\mathbf G}_{m}}}

\newcommand{\MW}{\mathrm{MW}}

\newcommand{\ho}[1]{\mathcal{H}({#1})}
\newcommand{\hop}[1]{\mathcal{H}_{\bullet}({#1})}

\newcommand{\bpi}{\boldsymbol{\pi}}
\newcommand{\piaone}{{\bpi}^{\aone}}

\newcommand{\Nis}{\operatorname{Nis}}

\newcommand{\Shv}{{\mathscr{Shv}}}
\newcommand{\Sm}{\mathscr{Sm}}

\newcommand{\Spc}{\mathscr{Spc}}

\newcommand{\Ab}{\mathscr{Ab}}

\newcommand{\K}{{{\mathbf K}}}

\newcommand{\hsnis}{\mathcal{H}_s^{\Nis}(k)}
\newcommand{\hspnis}{\mathcal{H}_{s,\bullet}^{\Nis}(k)}

\newcommand{\F}{{\mathcal F}}

\newcounter{intro}
\theoremstyle{plain}
\newtheorem{thm}{Theorem}[section]

\newtheorem{lem}[thm]{Lemma}
\newtheorem{cor}[thm]{Corollary}
\newtheorem{prop}[thm]{Proposition}
\newtheorem*{claim*}{Claim}  

\newtheorem{question}[thm]{Question}

\newtheorem*{thm*}{Theorem}
\newtheorem*{problem*}{Problem}

\newtheorem{thmintro}{Theorem}

\theoremstyle{definition}
\newtheorem{defn}[thm]{Definition}

\theoremstyle{remark}
\newtheorem{rem}[thm]{Remark}
\newtheorem{remintro}[thmintro]{Remark}

\newtheorem{ex}[thm]{Example}

\numberwithin{equation}{section}

\begin{document}
\pagestyle{fancy}
\renewcommand{\sectionmark}[1]{\markright{\thesection\ #1}}
\fancyhead{}
\fancyhead[LO,R]{\bfseries\footnotesize\thepage}
\fancyhead[LE]{\bfseries\footnotesize\rightmark}
\fancyhead[RO]{\bfseries\footnotesize\rightmark}
\chead[]{}
\cfoot[]{}
\setlength{\headheight}{1cm}

\author{\begin{small}Aravind Asok\thanks{Aravind Asok was partially supported by National Science Foundation Awards DMS-0900813 and DMS-0966589.}\end{small} \\ \begin{footnotesize}Department of Mathematics\end{footnotesize} \\ \begin{footnotesize}University of Southern California\end{footnotesize} \\ \begin{footnotesize}Los Angeles, CA 90089-2532 \end{footnotesize} \\ \begin{footnotesize}\url{asok@usc.edu}\end{footnotesize} \and \begin{small}Jean Fasel\thanks{Jean Fasel was supported by the Swiss National Science Foundation, grant PAOOP2\_129089}\end{small} \\ \begin{footnotesize}Fakult\"at Mathematik\end{footnotesize} \\ \begin{footnotesize} Universit\"at Duisburg-Essen, Campus Essen \end{footnotesize} \\ \begin{footnotesize}Thea-Leymann Strasse 9, D-45127 Essen\end{footnotesize} \\ \begin{footnotesize}\url{jean.fasel@gmail.com}\end{footnotesize}}

\title{{\bf Algebraic vector bundles on spheres}}
\date{}
\maketitle


\begin{abstract}
We determine the first non-stable ${\mathbb A}^1$-homotopy sheaf of $SL_n$.  Using techniques of obstruction theory involving the ${\mathbb A}^1$-Postnikov tower, supported by some ideas from the theory of unimodular rows, we classify vector bundles of rank $\geq d-1$ on split smooth affine quadrics of dimension $2d-1$.  These computations allow us to answer a question posed by Nori, which gives a criterion for completability of certain unimodular rows.  Furthermore, we study compatibility of our computations of ${\mathbb A}^1$-homotopy sheaves with real and complex realization.
\end{abstract}

\begin{footnotesize}
\tableofcontents
\end{footnotesize}

\section{Introduction}

This paper, which is part of a collection involving \cite{Asok12} and \cite{Asok12c} is devoted to using the techniques of the Morel-Voevodsky $\aone$-homotopy theory \cite{MV}, in conjunction with ideas from the obstruction theory in topology, to understand algebraic vector bundles on smooth affine schemes.  That this can be done, relies on F. Morel's proof of an algebro-geometric analog of the classical Steenrod representability theorem \cite{Steenrod} for topological vector bundles on spaces having the homotopy type of a CW complex.  More precisely, Morel identifies the set of isomorphism classes of vector bundles of a fixed rank on a smooth {\em affine} scheme $X$ as the set of morphisms in the $\aone$-homotopy category from $X$ to an algebro-geometric model of an appropriate infinite Grassmannian \cite{MField}.


Henceforth, fix a field $k$ that is assumed to be perfect and to have characteristic unequal to $2$.  The former requirement is necessitated by a number of foundational results quoted from \cite{MField}, while the latter assumption is imposed by our appeal to the machinery of Grothendieck-Witt theory.  At several points, it will also be necessary to assume that $k$ is infinite (e.g., this is necessary in Section \ref{s:unimodular}), but the reader may want to assume this from the beginning since in \cite[Lemma 1.15]{MField} Morel appeals to a form of Gabber's presentation lemma that was personally communicated to him by Gabber; the published version of this result requires $k$ to be infinite.

The ``spheres" we refer to in the title are the smooth quadric hypersurfaces $Q_{2n-1}$ in ${\mathbb A}^{2n}$ defined, for any integer $n \geq 1$ by the equation $\sum_{i=1}^n x_i y_i = 1$.  These varieties are spheres in the sense of $\aone$-homotopy theory.  Projecting onto $x_1,\ldots,x_n$, the quadric $Q_{2n-1}$ admits a morphism to ${\mathbb A}^n \setminus 0$ that is Zariski locally trivial and has affine space fibers; as a consequence the projection morphism is an isomorphism in the Morel-Voevodsky $\aone$-homotopy category $\ho{k}$ \cite{MV}.  This isomorphism can be used to identify the $\aone$-homotopy type of $Q_{2n-1}$ as $\Sigma^{n-1}_s \mathbf{G}_m^{\wedge n}$, i.e., $Q_{2n-1}$ is a smooth affine model of an $\aone$-homotopy sphere (see, e.g., \cite[\S 3 Example 2.20]{MV}).  These quadrics are important because they have the simplest non-trivial $\aone$-homotopy type, and also because of their connection with the classical theory of unimodular rows, for which they provide ``universal" examples \cite{Raynaud}.

One goal of this paper is to understand algebraic vector bundles on ``spheres" in the above sense.  For any integer $r \geq 1$, one can form classifying spaces $BGL_r$ (resp. $BSL_r$) of the algebraic groups $GL_r$ (resp. $SL_r$) \cite[\S 4]{MV}.  If $X$ is a smooth affine scheme, F. Morel \cite[Theorem 1.29]{MField} identifies the set of {\em free} $\aone$-homotopy classes of maps $[X,BGL_r]_{\aone}$ with the set ${\mathcal V}_r(X)$ of isomorphism classes of rank $r$ vector bundles on $X$; similarly, $[X,BSL_r]_{\aone}$ corresponds to isomorphism classes of rank $r$ vector bundles equipped with a fixed trivialization of the determinant (we call these ``oriented vector bundles"; see Theorem \ref{thm:orientedvectorbundles} for details).\footnote{Morel states this result only for $r \neq 2$, but as he remarks later, recent work of L.-F. Moser \cite{Moser} allows one to treat the case $r = 2$ as well.  Nevertheless, this fact is not essential to anything in this paper.}  For a variety with trivial Picard group, any vector bundle can be oriented.  On such a variety, the map from the set of isomorphism classes of oriented vector bundles to the set of isomorphism classes of vector bundles given by ``forgetting orientations" is surjective. In general, however, the map ``forgetting orientations" is neither injective nor surjective; an illustration of this fact is given in Section \ref{s:unimodular}.

Now, F. Morel has shown that $BSL_r$ is simply connected from the standpoint of $\aone$-homotopy theory \cite[Theorem 7.20]{MField}.  Just as in classical algebraic topology, the natural map from {\em pointed} $\aone$-homotopy classes of maps with target $BSL_r$ to {\em free} $\aone$-homotopy classes given by ``forgetting the base-point" is a bijection.  Write ${\mathscr V}^o_r(X)$ for the set of isomorphism classes of oriented rank $r$ vector bundles on a smooth scheme $X$, we have constructed an identification
\[
{\mathscr V}^o_r(Q_{2n-1}) \isomto [S^{n-1}_s \wedge \gm^{\wedge n},(BSL_r,\ast)]_{\aone},
\]
where the object on the right hand side denotes pointed $\aone$-homotopy classes of maps.  Note also that, by definition, the object on the right hand side of the above isomorphism is the set of sections over $\Spec k$ of the $\aone$-homotopy sheaf $\bpi_{n-1,n}^{\aone}(BSL_r)$.  Thus, we have reduced the problem of describing oriented vector bundles on the aforementioned motivic spheres to the (necessarily difficult) task of computing $\aone$-homotopy groups of $BSL_r$.



Results of F. Morel identify the sheaf $\bpi_{n-1,n}^{\aone}(BSL_r)$ as the $n$-fold ``contraction" of the sheaf $\bpi_{n-1}^{\aone}(BSL_r)$ (see Section \ref{s:preliminaries} for discussion about contractions). When $n < r$, the sheaf $\bpi_{n}^{\aone}(BSL_r)$ is already ``in the stable range" in the sense that it coincides with the sheaf $\K^Q_{n}$, i.e., the sheafification for the Nisnevich topology on smooth varieties of the Quillen K-theory presheaf. In Section \ref{section:metastableSL2n}, we introduce a morphism of sheaves $\psi_n:\K^Q_n\to \K_n^M$ for any $n\geq 2$ where the right-hand term is the unramified Milnor K-theory sheaf; this morphism is the sheafification of a homomorphism between Quillen K-theory groups and Milnor K-theory groups introduced originally by Suslin in \cite{Suslin82b}. The cokernel $\mathbf{S}_n$ of $\psi_n$ is the crucial ingredient we need to describe the first non-stable $\aone$-homotopy sheaf of $BSL_r$, and one of our main tasks is to elucidate its structure.  To this end, we show that there exists an epimorphism of sheaves $\K^M_{n}/(n-1)! \to {\mathbf S}_n$ for any $n\geq 2$ and a morphism of sheaves ${\mathbf S}_n\to \K_{n}^M/2$ for any odd integer $n\geq 3$. These results allow us to define the sheaf $\mathbf{T}_n$ for any odd integer $n\geq 3$ as the fiber product
\[
\xymatrix{{\mathbf T}_{n}\ar[r]\ar[d] & {\bf I}^{n}\ar[d] \\
{\mathbf S}_{n}\ar[r] & \K_{n}^M/2;}
\]
here ${\mathbf I}^{n}$ is the unramified sheaf corresponding to the $n$-th power of the fundamental ideal in the Witt ring.   With this additional notation, our computation of the first non-stable $\aone$-homotopy sheaf of $BSL_r$ can be summarized as follows.

\begin{thmintro}[See Theorem \ref{thm:homotopysheafevencase}]
\label{thmintro:mainhomotopysheaf}
If $k$ is a perfect field (having characteristic unequal to $2$), for any integer $m \geq 1$, there are short exact sequences of strictly $\aone$-invariant sheaves of the form
\[
\begin{split}
0 \longrightarrow {\mathbf T}_{2m+1} \longrightarrow &\bpi_{2m}^{\aone}(BSL_{2m}) \longrightarrow \K^Q_{2m} \longrightarrow 0, \\
0 \longrightarrow {\mathbf S}_{2m+2} \longrightarrow &\bpi_{2m+1}^{\aone}(BSL_{2m+1}) \longrightarrow \K^Q_{2m+1} \longrightarrow 0.
\end{split}
\]
\end{thmintro}


Oriented algebraic vector bundles of rank $\geq n$ on $Q_{2n-1}$ are easy to describe, essentially because the classification problem for such bundles already lies in the ``stable range."  Thus, we turn our attention to oriented vector bundles of rank $n-1$, which we refer to as the ``critical" rank.  While Theorem \ref{thmintro:mainhomotopysheaf} does not immediately provide enough information to completely describe the set of isomorphism classes of oriented rank $n-1$ vector bundles on $Q_{2n-1}$, it does reduce the problem to understanding contractions of ${\mathbf S}_n$.  Moreover, upon $n$-fold contraction, the problem of providing an explicit description of ${\mathbf S}_n$ becomes in a sense geometric and, with some input from the theory of unimodular rows, we can then give a rather explicit classification of oriented rank $(n-1)$-vector bundles on $Q_{2n-1}$. 

\begin{thmintro}[See Theorems \ref{thm:bundleseven} and \ref{thm:bundlesodd}]
\label{thmintro:vectorbundlesonoddspheres}
If $k$ is an infinite perfect field having characteristic unequal to $2$, $n$ is an integer $\geq 1$, and $W(k)$ denotes the Witt group of $k$, there are canonical isomorphisms
\[
{\mathscr V}_{n-1}^o(Q_{2n-1}) \isomto \begin{cases}
\Z/(n-1)! & \text{ if } n = 2m \\
\Z/(n-1)! \times_{\Z/2} W(k) & \text{ if } n = 2m+1;
\end{cases}
\]
where the maps in the fiber product are the reduction modulo $2$ map and the rank homomorphism $W(k) \to \Z/2$.
\end{thmintro}

\begin{remintro}
Using this computation, we can determine the set of isomorphism classes of rank $(n-1)$-vector bundles on $Q_{2n-1}$: see Corollaries \ref{cor:nonorientedbundleseven} and \ref{cor:nonorientedbundlesodd} for precise statements.  Indeed, we will see that when $n = 2m$, vector bundles admit unique orientations, while when $n = 2m+1$, isomorphic vector bundles can admit different orientations.
\end{remintro}


In turn, our classification results can be applied to obtain new results about unimodular rows.  In \cite{SuslinStablyFree}, Suslin gave a condition that was sufficient to insure that a unimodular row (see Section \ref{s:unimodular} for some recollections about unimodular rows) over any ring $R$ can be completed to an invertible matrix over $R$.  In \cite{MohanKumarUnimodular}, Nori inquired about a possible generalization of Suslin's theorem. This generalization turned out to be essentially correct when the base field is algebraically closed \cite{MohanKumarUnimodular}.  In \cite{Fasel11c}, the second author constructed a counterexample to Nori's original question and proposed a refined version.  The computations of Theorem \ref{thmintro:vectorbundlesonoddspheres} can be used to answer this refined version of Nori's question.

\begin{thmintro}[See Theorems \ref{thm:Nori} and \ref{thm:Nori+}]
\label{thmintro:Nori}
Suppose $k$ is a field of characteristic different from $2$, $R=k[x_1,\ldots,x_n]$ is a polynomial ring in $n$ variables, $\phi:R\to A$ is a $k$-algebra homomorphism such that $\sum \phi(x_i)A=A$, and $f_1,\ldots,f_n$ are elements of $R$ such that the reduced subscheme of ${\mathbb A}^n$ defined by the ideal $I(f_1,\ldots,f_n)$ coincides with $0 \in {\mathbb A}^n$; write ${\mathbf f}: {\mathbb A}^n \setminus 0 \to {\mathbb A}^n \setminus 0$ for the morphism induced by $(f_1,\ldots,f_n)$.  Assume that $length(R/(f_1,\ldots,f_n))$ is divisible by $(n-1)!$.
\begin{itemize}
\item If $n$ is odd, then $(\phi(f_1),\ldots,\phi(f_n))$ is completable.
\item If $n$ is even, then one can attach an element $deg({\mathbf f}) \in W(k)$ to ${\mathbf f}$, and if $deg({\mathbf f}) = 0$, then $(\phi(f_1),\ldots,\phi(f_n))$ is completable.
\end{itemize}
\end{thmintro}

Theorem \ref{thmintro:vectorbundlesonoddspheres} also allows us to deduce Theorem \ref{thm:complexrealization}, which discusses compatibility of our homotopy sheaf computations with complex realization.  In a sense, these results ``explain" that the factors of $n!$ that appear in the homotopy sheaves arise from complex Bott periodicity, while the factors of ${\mathbf I}^{n}$ that appear arise because of real Bott periodicity.

\subsubsection*{Acknowledgements}
The first author would like to thank Brent Doran for many discussions about the $\aone$-homotopy theory of quadrics.  Both the authors wish to thank the referees for a very careful reading of the paper; the numerous comments/remarks helped to greatly improve the exposition.

\section{Preliminaries}
\label{s:preliminaries}
In this section, we begin by fixing notation regarding $\aone$-homotopy theory.  General references for this section (and for the remainder of the paper) are the foundational works \cite{MV} and \cite{MField}.  We establish some preliminary results regarding change of basepoints, and contracted homotopy sheaves that will be used in the computations throughout the paper.  We also review some results about Milnor-Witt K-theory, mainly notation and results from \cite{Morel04}, that will appear repeatedly in the remainder of the paper.

\subsubsection*{Notational preliminaries}
Assume $k$ is a field.  Write $\Sm_k$ for the category of schemes that are smooth, separated and have finite type over $\Spec k$ and $\Spc_k := \Delta^{\circ}\Shv_{\Nis}(\Sm_k)$ (resp. $\Spc_{k,\bullet}$) for the category of (pointed) simplicial sheaves on the site of smooth schemes equipped with the Nisnevich topology; objects of $\Spc_k$ (resp. $\Spc_{k,\bullet}$) will be referred to as {\em (pointed) $k$-spaces}, or simply as {\em (pointed) spaces} if $k$ is clear from context.  Write $\hsnis$ (resp. $\hspnis$) for the homotopy category of $\Spc_k$ equipped with the model structure where cofibrations are monomorphisms, weak equivalences are those morphisms that are stalkwise weak equivalences of simplicial sets, and fibrations are those morphisms having the right lifting property with respect to morphisms that are simultaneously cofibrations and weak equivalences; we refer to the category $\hsnis$ (resp. $\hspnis$) as the (pointed) simplicial homotopy category.  Write $\ho{k}$ (resp. $\hop{k}$) for the Morel-Voevodsky $\aone$-homotopy category \cite{MV} which is obtained as a localization, in the sense of Bousfield, of the simplicial homotopy category

Given two (pointed) spaces $\mathscr{X}$ and $\mathscr{Y}$, we set $[{\mathscr X},{\mathscr Y}]_s := \hom_{\hsnis}({\mathscr X},{\mathscr Y})$, and $[\mathscr{X},\mathscr{Y}]_{\aone} := \hom_{\ho{k}}(\mathscr{X},\mathscr{Y})$; morphisms in pointed homotopy categories will be denoted similarly with base-points explicitly specified.  We write $S^i_s$ for the constant sheaf on $\Sm_k$ associated with the simplicial $i$-sphere, and we view $\gm$ as a pointed space, pointed by $1$.  The $\aone$-homotopy sheaves of a pointed space $(\mathscr{X},x)$, denoted $\bpi_i^{\aone}(\mathscr{X},x)$, are defined as the Nisnevich sheaves associated with the presheaves $U \mapsto [S^i_s \wedge U_+,(\mathscr{X},x)]_{\aone}$.  As $\mathscr{X}(k)$ is non-empty, the forgetful map $[U_+,(\mathscr{X},x)]_{\aone} \to [U,\mathscr{X}]_{\aone}$ is a bijection, so we can write $\bpi_0^{\aone}(\mathscr{X})$ independent of the choice of base point $x$. We also write $\bpi_{i,j}^{\aone}({\mathscr X},x)$ for the Nisnevich sheafification of the presheaf $U \mapsto [S^i_s \wedge \mathbf{G}_m^{\wedge j} \wedge U_+,({\mathscr X},x)]_{\aone}$.

A presheaf of sets $\F$ on $\Sm_k$ is called {\em $\aone$-invariant} if for any smooth $k$-scheme $U$ the morphism $\F(U) \to \F(U \times \aone)$ induced by pullback along the projection $U \times \aone \to U$ is a bijection.  A Nisnevich sheaf of groups $\mathscr{G}$ is called {\em strongly $\aone$-invariant} if the cohomology presheaves $H^i_{\Nis}(\cdot,\mathscr{G})$ are $\aone$-invariant for $i = 0,1$.  A Nisnevich sheaf of abelian groups ${\mathbf A}$ is called {\em strictly $\aone$-invariant} if the cohomology presheaves $H^i_{\Nis}(\cdot,{\mathbf A})$ are $\aone$-invariant for every $i \geq 0$. Note that strongly $\aone$-invariant sheaves of abelian groups are strictly $\aone$-invariant by \cite[Corollary 5.45]{MField}. Henceforth, unless otherwise indicated, the word sheaf will mean Nisnevich sheaf on $\Sm_k$, and the undecorated symbol $H^i$ will mean ``$i$-th cohomology (of a sheaf) with respect to the Nisnevich topology."

If $n \geq 0$ is an integer, a space $\mathscr{X}$ is called $\aone$-$n$-connected if $\bpi_0^{\aone}(\mathscr{X}) = *$, and, for any choice of base-point $x \in \mathscr{X}(k)$ and any integer $i \leq n$, $\bpi_i^{\aone}(\mathscr{X},x) = 0$.  If $G$ is an algebraic group and we view $G$ as a pointed space, the base-point is always the identity section $\Spec k \to G$ and for this reason will usually be suppressed.  Likewise, the space $BG$, defined by means of the simplicial bar construction \cite[\S 4]{MV}, has a canonical base-point corresponding to the unique $0$-simplex, and this will usually be suppressed from notation as well.

\subsubsection*{Change of basepoints}
The following lemma is an analog in $\aone$-homotopy theory of a classical result regarding the relationship between pointed and free homotopy classes when the target is simply connected (see, e.g., \cite[Chapter 7.3 Corollary 7 p. 383]{Spanier} for the classical result).


\begin{lem}
\label{lem:basepoint}
Let $({\mathscr X},x_0)$ and $({\mathscr Y},y_0)$ be pointed spaces. Suppose that $\mathscr{Y}$ is $\aone$-1-connected. Then the map $f:[(\mathscr{X},x_0),(\mathscr{Y},y_0)]_{\aone}\to [\mathscr{X},\mathscr{Y}]_{\aone}$ obtained by forgetting base-points is a bijection.
\end{lem}

\begin{proof}
We can assume without loss of generality that ${\mathscr Y}$ is simplicially fibrant and $\aone$-local.  Since all spaces are cofibrant (either simplicially or in the $\aone$-local model structure), we have $[{\mathscr X},{\mathscr Y}]_s = [{\mathscr X},{\mathscr Y}]_{\aone}$.  Furthermore, $[{\mathscr X},{\mathscr Y}]_{\aone}$ is the quotient of the set of morphisms $\hom_{\Spc_k}({\mathscr X},{\mathscr Y})$ by the relation of simplicial homotopy.  More precisely, write $\Delta^1_s$ for the simplicial unit interval, which has two $0$-simplices $0$ and $1$.  We say that two morphisms $h_0,h_1: {\mathscr X} \to {\mathscr Y}$ are simplicially homotopic if there exists a morphism $H: {\mathscr X} \times \Delta^1_s \to {\mathscr Y}$ whose specializations to $0$ and $1$ coincide with $h_0$ and $h_1$ respectively.


Note that, since ${\mathscr Y}$ fibrant and $\aone$-local, the assumption that ${\mathscr Y}$ is $\aone$-$1$-connected is equivalent to ${\mathscr Y}$ being simplicially $1$-connected.

We first prove that $f$ is surjective. Suppose $h:{\mathscr X}\to {\mathscr Y}$ is a morphism. Since ${\mathscr Y}$ is simplicially connected, there exists a simplicial homotopy between $h(x_0)$ and $y_0$, i.e., there exists a morphism $H:\Delta^1_s\to {\mathscr Y}$ such that $\Delta^1_s(1)=h(x_0)$ and $\Delta^1_s(0)=y_0$. We get a commutative diagram
\[
\xymatrix{\star\ar[r]^-H\ar[d]_{x_0} & \underline{\mathrm{Hom}}(\Delta^1_s,{\mathscr Y})\ar[d] \\
{\mathscr X}\ar[r]_-h & {\mathscr Y}}
\]
where the right-hand vertical map is the evaluation map $G\mapsto G(1)$, the left vertical map is the inclusion of $x_0$.  The evaluation map is a simplicial weak equivalence and, since ${\mathscr Y}$ is simplicially fibrant, also a fibration.  Since ${\mathscr X}$ is cofibrant, the left lifting property guarantees the existence of a map $F: {\mathscr X} \to \underline{\hom}(\Delta^1_s,{\mathscr Y})$ making the two resulting triangles commute, i.e., $F(x_0) = H$, while $ev_1 \circ F = y_0$ and this second composite provides the required lift.

If $h_0,h_1:({\mathscr X},x_0)\to ({\mathscr Y},y_0)$ are homotopic morphisms, then by definition there exists a morphism $H:{\mathscr X}\times \Delta^1_s\to Y$ such that $H(0)=h_0$ and $H(1)=h_1$. Now $H(x_0,\_):\Delta^1_s\to {\mathscr Y}$ is a path based at $y_0$. By assumption, there exists $G:\Delta^1_s\times\Delta^1_s\to Y$ such that $G(1)=H(x_0,\_)$ and $G(0)=y_0$. We obtain a commutative diagram
\[
\xymatrix{({\mathscr X}\times \{0\}) \cup (\{x_0\}\times\Delta^1_s)\cup ({\mathscr X}\times \{1\})\ar[d]\ar[r]^-N & \underline{\mathrm{Hom}}(\Delta^1_s,{\mathscr Y})\ar[d] \\
{\mathscr X}\times\Delta^1_s\ar[r]_-H & {\mathscr Y}}
\]
where $N$ is defined by $N_{\vert {\mathscr X}\times \{0\}}=h_0$, $N_{\vert \{x_0\}\times\Delta^1}=G$ and $N_{\vert {\mathscr X}\times \{1\}}=h_1$. The left vertical map is a cofibration, while the right vertical map is a fibration and weak equivalence. It follows that there exists $M:{\mathscr X}\times\Delta^1_s\to \underline{\mathrm{Hom}}(\Delta^1,{\mathscr Y})$ making the whole diagram commutative. Composing with the evaluation at $0$ map $\underline{\mathrm{Hom}}(\Delta^1,{\mathscr Y})\to {\mathscr Y}$, we obtain the required base point preserving homotopy between $h_0$ and $h_1$.
\end{proof}



\subsubsection*{Contracted homotopy sheaves}
Recall that if ${\mathbf A}$ is a strictly $\aone$-invariant sheaf, one defines the contracted sheaf ${\mathbf A}_{-1}$ by means of the formula ${\mathbf A}_{-1}(U) = \ker(s^*: {\mathbf A}(\gm \times U) \to {\mathbf A}(U))$, where $s: U \to \gm \times U$ is the map $1 \times id_U$ coming from the identity section $1$ of $\gm$.  The contraction construction is an exact functor and preserves the property of being strictly $\aone$-invariant (see, e.g., \cite[Lemma 7.33]{MField} or, more precisely, its proof).  One then defines the $i$-fold contracted sheaf ${\mathbf A}_{-i}$ inductively by ${\mathbf A}_{-i} = ({\mathbf A}_{-i+1})_{-1}$.  The main import of the contraction construction lies in the following result of F. Morel, which we will use repeatedly in the sequel.

\begin{thm}[{\cite[Theorem 6.13]{MField}}]
\label{thm:contractions}
If $({\mathscr X},x)$ is a pointed $\aone$-connected space, then there are canonical isomorphisms for any $i\geq 1$ and $j\geq 0$
\[
\bpi_{i,j}^{\aone}({\mathscr X},x) \cong \bpi_i^{\aone}({\mathbf R}\Omega^j_{\gm}{\mathscr X}) \cong \bpi_i^{\aone}({\mathscr X},x)_{-j};
\]
here ${\mathbf R}\Omega^j_{\gm}(\cdot)$ is the (derived) iterated $\gm$-loop space functor.
\end{thm}

A convenient summary of calculations of contractions used in this paper, and other basic properties of the contraction construction are presented below.

\begin{cor}
\label{cor:tatehomotopysheavesofgln}
If $j \geq 0$, $i \geq 1$, and $n \geq 2$ are integers, then there are canonical isomorphisms $\bpi_{i,j}^{\aone}(GL_n)\cong \bpi_{i}^{\aone}(GL_n)_{-j}$.
\end{cor}

\begin{proof}
We use the fiber sequence $GL_n \to EGL_n \to BGL_n$ discussed in the next section.  Since $EGL_n$ is $\aone$-contractible, there are isomorphisms $\bpi_{i,j}^{\aone}(GL_n) \cong \bpi_{i+1,j}^{\aone}(BGL_n)$, and $BGL_n$ is $\aone$-connected.  We can then apply the aforementioned result of F. Morel to finish the proof (note: we cannot apply the aforementioned result directly to $GL_n$ since it fails to be $\aone$-connected).
\end{proof}

\begin{ex}[Contracted homotopy sheaves of $B\gm$]\label{ex:gm}
The space $\gm$ is fibrant as a simplicial sheaf and also $\aone$-local.  As a consequence, we see that $\bpi_0^{\aone}(\gm) = \gm$, while the higher $\aone$-homotopy sheaves of $\gm$ are necessarily trivial for {\em any} choice of base point, i.e., $\gm$ is discrete from the standpoint of $\aone$-homotopy theory.  There is an $\aone$-fiber sequence
\[
\gm \longrightarrow E\gm \longrightarrow B\gm;
\]
in fact, $E\gm$ is an $\aone$-covering space of $B\gm$ in the sense of \cite[\S 7.1]{MField}. Since $E\gm$ is $\aone$-contractible, there are thus identifications $\bpi_{i}^{\aone}(\gm) \cong \bpi_{i+1}^{\aone}(B\gm)$ for any $i\geq 0$. Note moreover that $B\gm$ is $\aone$-connected and we can thus apply Theorem \ref{thm:contractions} to compute the contractions of its homotopy sheaves. We find
\[
\bpi_{i,j}^{\aone}(B\gm)=\begin{cases} \gm & \text{if $i=1$ and $j=0$.} \\
\Z & \text{if $i=1$ and $j=1$.} \\
0 & \text{otherwise.}
\end{cases}
\]
\end{ex}

\begin{ex}
\label{ex:pi0ofGLn}
The canonical inclusion $SL_n \to GL_n$ (as the kernel of the determinant homomorphism) fits into an $\aone$-fiber sequence (see the next section) of the form
\[
\gm \longrightarrow BSL_n \longrightarrow BGL_n.
\]
The space $BSL_n$ is $\aone$-simply connected; this follows because $SL_n$ is $\aone$-connected and (as we will see in Section \ref{section:metastableSL2n}) there is an $\aone$-fiber sequence of the form
\[
SL_n \longrightarrow ESL_n \longrightarrow BSL_n.
\]
Therefore, we can identify $\bpi_1^{\aone}(BGL_n) \cong \bpi_0^{\aone}(GL_n) \cong \gm$.  Since we can identify $\bpi_{i,j}^{\aone}(GL_n) \cong \bpi_{i+1,j}^{\aone}(BGL_n)$, we see that $\piaone_{0,j}(GL_n) = \bpi_{1,j}^{\aone}(BGL_n)$, and the latter sheaf is $\gm$ if $j = 0$, $\Z$ if $j = 1$ and is trivial for $j > 1$.
\end{ex}

\subsubsection*{Some exact sequences involving Milnor-Witt K-theory sheaves}
Write $K^{MW}_*(k)$ for the graded Milnor-Witt K-theory ring.  Recall that $K^{MW}_*(k)$ is generated by symbols $[a]$ with $a\in k^{\times}$ of degree $+1$ and a symbol $\eta$ of degree $-1$ satisfying various relations \cite[Definition 5.1]{Morel04}.  Write $I^*(k)$ for the graded ring corresponding to the powers of the fundamental ideal in the Witt ring; recall that $I^m(k)$ is additively generated by the classes of $m$-fold Pfister forms for $m\geq 1$ and that $I^m(k)=W(k)$ for $m\leq 0$.  Assigning to a symbol $a \in k^{\times}$ the class of the Pfister form $\langle \langle a \rangle \rangle$ defines a group homomorphism $K^{MW}_1(k) \to I^1(k)$; this homomorphism extends to a graded ring homomorphism $K^{MW}_*(k) \to I^*(k)$.  Likewise, if $K^M_*(k)$ denotes the $\Z$-graded Milnor K-theory ring (here $K^M_{i}(k) = 0$ if $i < 0$), there is also a homomorphism of graded rings $K^{MW}_*(k) \to K^M_*(k)$ that sends $\eta$ to $0$.

Let $k_*(k) = K^M_*(k)/2K^M_*(k)$ (we beg the reader's indulgence for this unfortunate choice of notation, which will persist only through this paragraph).  There is a canonical homomorphism of graded rings $K^M_*(k) \to k_*(k)$ and a homomorphism $K^M_*(k)\to I^*(k)/I^{*+1}(k)$ mapping pure symbols of the shape $\{a_1,\ldots,a_n\}$ to the Pfister form $\langle \langle a_1,\ldots,a_n\rangle\rangle$.  This homomorphism factors through $k_*(k)$ and induces an isomorphism $k_*(k)\isomt I^*(k)/I^{*+1}(k)$ by the Milnor conjecture on quadratic forms \cite{OVV}. Composing the projection $I^*(k)\to I^*(k)/I^{*+1}(k)$ with the inverse of this isomorphism defines a homomorphism of graded rings $I^*(k) \to k_*(k)$.  Morel \cite[Theorem 5.3]{Morel04} shows that these various homomorphisms fit into a cartesian square of graded rings of the form
\[
\xymatrix{
K^{MW}_*(k) \ar[r]\ar[d] & K^M_*(k) \ar[d] \\
I^*(k) \ar[r] & k_*(k).
}
\]

The above square can be sheafified in an appropriate sense: the objects and morphisms in the fiber square are compatible with residue maps and yield a cartesian square of unramified sheaves of graded rings
\begin{equation}\label{eqn:Milnor-Witt}
\xymatrix{
\K^{MW}_* \ar[r]\ar[d] & \K^M_*\ar[d] \\
{\mathbf I}^*\ar[r] & \K^M_*/2.
}
\end{equation}
We refer the reader to \cite[\S 2.2-4]{MMilnor} for a detailed discussion of the unramified Milnor K-theory sheaf $\K^M_n$, the unramified sheaf ${\mathbf I}^m$ and the homomorphism ${\mathbf I}^* \to \K^M_*/2$, which Morel calls a sheafification of Milnor's homomorphism.  We refer the reader to \cite[\S 3]{MField} for the construction of the sheaf $\K^{MW}_n$ and the homomorphism in the left hand column and the top row.  Because the above diagram is cartesian, one deduces immediately the existence of the following exact sequences.

\begin{prop}
\label{prop:milnorwittexactsequences}
For every $n\in\Z$, there are short exact sequences of the form
\[
0 \longrightarrow {\mathbf I}^{n+1} \longrightarrow \K^{MW}_n \longrightarrow \K^M_n \longrightarrow 0,
\]
and
\[
0 \longrightarrow 2\K^M_n \longrightarrow \K^{MW}_n \longrightarrow {\mathbf I}^{n} \longrightarrow 0.
\]
Moreover, the map $\K^{MW}_n \to \K^{MW}_{n-1}$ induced by multiplication by $\eta$ factors as a composite $\K^{MW}_{n} \to {\mathbf I}^{n} \to \K^{MW}_{n-1}$, where the two constituent maps are those in the above exact sequences.
\end{prop}

\begin{proof}
The only thing that remains to be checked is the final statement.  To that end, the map $K^{MW}_n(k) \to I^n(k)$ is defined by sending a symbol $[a_1]\cdots[a_{n}]$ to $\langle\langle a_1,\ldots,a_{n} \rangle\rangle$, and the map $I^n(k) \to K^{MW}_{n-1}(k)$ is defined by sending a Pfister form $\langle\langle a_1,\ldots,a_n \rangle \rangle$ to $\eta [a_1]\cdots[a_{n}]$.  We refer the reader to \cite[\S 5]{Morel04} for more details.
\end{proof}

\subsubsection*{Further results on contractions}
Write $\K^Q_i$ for the Nisnevich sheaf associated with the presheaf $U \mapsto K_i(U)$, where $K_i$ denotes Quillen K-theory; these sheaves are called Quillen K-theory sheaves.

\begin{lem}
\label{lem:contractionsofquillenandmilnor}
For any integers $i,j \geq 0$ and any integer $n > 0$, there are canonical isomorphisms
\[
(\K^M_i)_{-j} \cong \begin{cases}\K^{M}_{i-j} & \text{ if } j \leq i \\ 0 & \text{ if } j > i \end{cases}, \text{ and}
\]
\[
(\K^Q_i)_{-j} \cong \begin{cases}\K^{Q}_{i-j} & \text{ if } j \leq i \\ 0 & \text{ if } j > i \end{cases}.
\]
\end{lem}

\begin{proof}[Remarks on the proof.]
The proofs of these two statements can be obtained by the same method as that of Proposition \ref{prop:contractionsofMW}, which is a bit more delicate, so we will prove that statement instead.
\end{proof}

\begin{rem}
Exactness of the contraction construction and Lemma \ref{lem:contractionsofquillenandmilnor} imply that $(\K^M_i/n)_{-j} \cong \K^M_{i-j}/n$ and $(\K^Q_i/n)_{-j} \cong \K^Q_{i-j}/n$ for any integer $n$.
\end{rem}

\begin{prop}
\label{prop:contractionsofMW}
For any integer $j$, there are canonical isomorphisms $({\mathbf I}^n)_{-j} \cong {\mathbf I}^{n-j}$ and $(\K^{MW}_n)_{-j} \cong \K^{MW}_{n-j}$.
\end{prop}

\begin{proof}
We first prove the statement for the sheaf ${\mathbf I}^n$. Let $X=\Spec A$, where $A$ is a local ring. In order to prove that $({\mathbf I}^n)_{-1}={\mathbf I}^{n-1}$, it suffices to prove that we have an exact sequence
\[
0 \longrightarrow {\mathbf I}^n(X) \longrightarrow {\mathbf I}^n(X\times\gm) \longrightarrow {\mathbf I}^{n-1}(X) \longrightarrow 0.
\]
Now ${\mathbf I}^n$ is in particular a Zariski sheaf for any $n\in\Z$ and the Gersten-Witt complex (filtered by the powers of the fundamental ideal) described for instance in \cite[\S 1]{Fasel07} provides a flasque resolution of this sheaf. The long exact sequence associated to the open embedding $X\times\gm\subset X\times\A^1$ together with the fact that $X$ is local reads then as
\[
0 \longrightarrow {\mathbf I}^n(X) \longrightarrow {\mathbf I}^n(X\times\gm) \longrightarrow H^1_{Zar, X\times\{0\}}(X\times\A^1,{\mathbf I}^{n}) \longrightarrow 0.
\]
We can use d\'evissage as described in \cite[Proposition 3.30]{Fasel07} (together with the obvious trivialization of the normal sheaf of $X\times\{0\}$ in $X\times\A^1$) to get an isomorphism of groups ${\mathbf I}^{n-1}(X)\to H^1_{Zar, X\times\{0\}}(X\times\A^1,{\mathbf I}^{n})$ as required.

To prove the corresponding statement for $\K^{MW}_n$, we first observe that all the sheaves involved in the Cartesian square (\ref{eqn:Milnor-Witt}) are strictly $\A^1$-invariant (see \cite{Panin10} for the sheaf ${\mathbf I}^n$). Contracting the exact sequence
\[
0 \longrightarrow \K^{MW}_n \longrightarrow \K^M_n\oplus {\mathbf I}^n \longrightarrow \K^m_n/2 \longrightarrow 0
\]
$j$ times, we then obtain an exact sequence
\[
0 \longrightarrow (\K^{MW}_n)_{-j} \longrightarrow (\K^M_n)_{-j}\oplus ({\mathbf I}^n)_{-j} \longrightarrow (\K^m_n/2)_{-j} \longrightarrow 0
\]
The result follows then from Lemma \ref{lem:contractionsofquillenandmilnor} together with the result for ${\mathbf I}^n$.
\end{proof}

\section{The first non-stable homotopy sheaf of $SL_{n}$ }\label{section:metastableSL2n}
The goal of this section is to compute the group $\bpi_{n-1}^{\aone}(SL_{n})$ for $n\geq 2$.  The answer is different depending on whether $n$ is odd or even.  To perform this computation, we use the theory of fiber sequences in $\aone$-homotopy, which was developed in \cite{MField} and extended by \cite{Wendt}; a convenient summary of the necessary results can be found in \cite[\S 2]{AsokPi1}.  In particular, we use the fact that $SL_n$-torsors (or $GL_n$-torsors), and their ``associated fiber spaces" give rise to $\aone$-fiber sequences; the precise statements we use can be found in \cite[Proposition 5.1, Proposition 5.2, and Theorem 5.3]{Wendt}.  We also use the fact that associated with an $\aone$-fiber sequence is a corresponding long exact sequence in $\aone$-homotopy sheaves (see, e.g., \cite[Lemma 2.10]{AsokPi1}).  A description of the connecting homomorphism in the long exact sequence is given in \cite[\S 3.7]{AsokPi1}.

\subsubsection*{Computations in the stable range}
We quickly review the computation of the homotopy sheaves of $SL_n$ and $GL_n$ in the stable range, which is due to Morel.  By \cite[\S 4 Theorem 3.13]{MV}, there is a space $\Z \times BGL_{\infty}$ that represents Quillen K-theory for smooth $k$-schemes (one can use an infinite Grassmannian as a particular model for $BGL_{\infty}$).  The next results describe the $\aone$-homotopy sheaves of $SL_n$ or $GL_n$ in the stable range in terms of Quillen K-theory.

\begin{lem}
For any $i\geq 0$ and any $n\geq 1$, we have isomorphisms $\bpi_i^{\aone}(GL_n)\simeq \bpi_{i+1}^{\aone}(BGL_n)$ and $\bpi_i^{\aone}(SL_n)\simeq \bpi_{i+1}^{\aone}(BSL_n)$. Moreover, the natural map $BSL_n\to BGL_n$ induces isomorphisms $\bpi_{i+1}^{\aone}(BSL_n)\simeq \bpi_{i+1}^{\aone}(BGL_n)$ for any $i\geq 1$.
\end{lem}

\begin{proof}
For $G=GL_n, SL_n$, we can use the explicit geometric models of $BG$ described in \cite[\S 4.2]{MV} to get a fiber sequence
\[
G \longrightarrow EG \longrightarrow BG
\]
where $EG$ is an $\aone$-contractible space. The long exact sequence of homotopy sheaves associated to this sequence yields the first statement. For the second statement, observe that we have a fiber sequence
\[
\gm \longrightarrow BSL_n \longrightarrow BGL_n.
\]
The connecting homomorphism ${\mathbf R}\Omega^1_s BGL_n \cong GL_n \to \gm$ is precisely the determinant map, and this is split by the inclusion $\gm \to GL_n$ given by $t \mapsto diag(t,1,\ldots,1)$.  Now, $BSL_n$ is $\aone$-simply connected since $SL_n$ is $\aone$-connected.  Since the space $\gm$ is $\aone$-rigid in the sense of \cite[\S 4 Example 2.4]{MV}, we have $\bpi_i^{\aone}(\gm) = 1$ for $i \geq 1$ (see also Example \ref{ex:gm}).  The inclusion $\gm \hookrightarrow GL_n$ induces an isomorphism $\gm = \bpi_0^{\aone}(\gm) \isomto \bpi_0^{\aone}(GL_n)$ as mentioned in Example \ref{ex:pi0ofGLn}.  Combining these observations gives the isomorphism $\bpi_i^{\aone}(SL_n) \cong \bpi_i^{\aone}(GL_n)$ for $i \geq 1$ and any $n$.
\end{proof}

\begin{thm}
\label{thm:stabilizationsln}
Suppose $i > 0$ and $n > 1$ are integers.  The morphisms
\[
\bpi_i^{\aone}(SL_{n-1}) \longrightarrow \bpi_i^{\aone}(SL_n)
\]
are epimorphisms for $i \leq n-2$ and isomorphisms for $i \leq n-3$. Furthermore, we have $\bpi_i^{\aone}(SL_n)\simeq \K_{i+1}^Q$ if $1\leq i\leq n-2$.
\end{thm}

\begin{proof}
By \cite[Proposition 8.12]{MField}, we have a fiber sequence
\[
SL_{n-1} \longrightarrow SL_{n} \longrightarrow SL_{n}/SL_{n-1}.
\]
Moreover, the map $SL_{n} \to {\mathbb A}^{n} \setminus 0$ given by projection onto the first column factors through a morphism $SL_{n}/SL_{n-1} \to {\mathbb A}^n \setminus 0$; the latter morphism is Zariski locally trivial with fibers isomorphic to ${\mathbb A}^{n-1}$ and is therefore an isomorphism in $\ho{k}$.  Now, \cite[Theorem 6.39]{MField} shows that ${\mathbb A}^n \setminus 0$ is $\aone$-$(n-2)$-connected, and the first assertion then follows from the long exact sequence in $\aone$-homotopy sheaves associated with an $\aone$-fibration.

Finally, by representability of algebraic K-theory, $\K^Q_{i+1}$ can also be described as $\bpi_{i+1}^{\aone}(\Z \times BGL_{\infty})$.  Moreover, for $i \geq 0$, the only contribution to this sheaf comes from the $\aone$-connected component of the base-point, so $\bpi_{i+1}^{\aone}(\Z \times BGL_{\infty}) \cong \bpi_{i+1}^{\aone}(BGL_{\infty})$. Now the above lemma shows that $\bpi_{i+1}^{\aone}(BSL_{\infty})=\bpi_{i+1}^{\aone}(BGL_{\infty})$ and the result follows.
\end{proof}

\subsubsection*{A short exact sequence describing $\bpi_{n-2}^{\aone}(SL_{n-1})$, $n \geq 3$}
We now consider the long exact sequence in $\aone$-homotopy sheaves associated with the $\aone$-fiber sequence
\[
SL_n/SL_{n-1}\longrightarrow BSL_{n-1}\longrightarrow BSL_{n}
\]
of \cite[Proposition 8.11]{MField}. This exact sequence takes the form
\[
\bpi_{n-1}^{\aone}(SL_n) \longrightarrow \bpi_{n-1}^{\aone}(SL_n/SL_{n-1}) \longrightarrow \bpi_{n-2}^{\aone}(SL_{n-1}) \longrightarrow \bpi_{n-2}^{\aone}(SL_n) \longrightarrow 0.
\]
We also observe that $\bpi_1^{\aone}(SL_2)$ and $\bpi_1^{\aone}(SL_3)$ are known to be sheaves of abelian groups and are therefore strictly $\aone$-invariant \cite[Corollary 6.2]{MField}, i.e., the sequence above is always a sequence of strictly $\aone$-invariant sheaves of groups. Furthermore, we know that $SL_n/SL_{n-1}$ is $\aone$ weak-equivalent to $\A^n\setminus 0$ and thus $\bpi_{n-1}^{\aone}(SL_n/SL_{n-1})\simeq \K^{\MW}_n$ by \cite[Theorem 6.40]{MField}.

Since the sheaves $\bpi_{n-2}^{\aone}(SL_n)$ are in the stable range by Theorem \ref{thm:stabilizationsln}, we can rewrite the exact sequence of the previous paragraph as
\[
\bpi_{n-1}^{\aone}(SL_n) \stackrel{q_{n-1}}{\longrightarrow} \K^{\MW}_n \stackrel{\delta_{n-1}}{\longrightarrow} \bpi_{n-2}^{\aone}(SL_{n-1}) \longrightarrow \K^Q_{n-1} \longrightarrow 0.
\]
Our goal is to understand the image of $\bpi_{n-1}^{\aone}(SL_n) \to \K^{\MW}_n$.  To this end, observe that we have a fiber sequence
\[
SL_{n+1}/SL_n\longrightarrow BSL_n\longrightarrow BSL_{n+1}
\]
which yields a diagram
\begin{equation}\label{equ:useful}
\xymatrix{ & & SL_{n+1}/SL_n\ar[d] \\
SL_n/SL_{n-1}\ar[r] & BSL_{n-1}\ar[r] & BSL_{n}\ar[d] \\
 & & BSL_{n+1}   }
\end{equation}
and we can consider the composite map $\Omega_s^1SL_{n+1}/SL_n\to \Omega_s^1BSL_n=SL_n\to SL_n/SL_{n-1}$. Applying $\bpi_{n-1}^{\aone}$, we get a map
\[
\xymatrix{\K_{n+1}^{MW}\ar[r]^-{\delta_n} & \bpi_{n-1}^{\aone}(SL_n)\ar[r]^-{q_{n-1}} & \K_n^{MW}}
\]
that we describe more precisely in the next section.

\subsubsection*{An Euler class computation}
Since $SL_{n+1}/SL_{n}$ is $\aone$-$(n-1)$-connected, if ${\mathbf A}$ is any strictly $\aone$-invariant sheaf, \cite[Theorem 3.30]{ADExcision} gives a canonical bijection
\[
H^{n}_{\Nis}(SL_{n+1}/SL_{n},{\mathbf A}) \isomto \hom_{\Ab^{\aone}_k}(\bpi_{n}^{\aone}(SL_{n+1}/SL_{n}),{\mathbf A}).
\]
Applying these observations with ${\mathbf A} = \K^{\MW}_{n}$, the morphism $q_{n-1} \circ \delta_{n}$ is determined by an element of $H^{n}_{\Nis}(SL_{n+1}/SL_{n},\K^{\MW}_{n})$.

The map $SL_{n+1}/SL_n\to BSL_n$ classifies the $SL_{n}$-torsor $SL_{n+1}\to SL_{n+1}/SL_n$. Recall that $Q_{2n+1}$ is the quadric hypersurface in $\A^{2n+2}$ defined by the equation $\sum_{i=1}^{n+1}x_iy_i=1$. We write $A_{2n+1}$ for the ring of global sections on $Q_{2n+1}$. Recall moreover that projecting a matrix to its first row and the first column of its inverse yields an isomorphism $SL_{n+1}/SL_n\simeq Q_{2n+1}$.

\begin{lem}
\label{lem:universaltorsor}
The vector bundle on $Q_{2n+1}$ corresponding to the torsor $SL_{n+1}\to SL_{n+1}/SL_n\simeq Q_{2n+1}$ is the total space of the stably free module of rank $n$ on $Q_{2n+1}$ defined by the (split) exact sequence
\[
\xymatrix@C=4em{0\ar[r] & P_{n}\ar[r] & (A_{2n+1})^{n+1}\ar[r]^-{(x_1,\ldots,x_{n+1})} & A_{2n+1}\ar[r] & 0.}
\]
\end{lem}

\begin{proof}
Let $V$ be the total space of $P_n$ over $Q_{2n+1}$. Explicitly, we have
\[
V=\Spec k[x_1,\ldots,x_{n+1},y_1,\ldots,y_{n+1},z_1,\ldots,z_{n+1}]/(\sum_{i=1}^{n+1} x_iy_i-1, \sum_{i=1}^{n+1} x_iz_i).
\]
Let
\[
f:\A^n\times SL_{n+1}\to V
\]
be defined by $(v,M)\mapsto (M^te_1, M^{-1}e_1, M^{-1}v^\prime)$ where $v^\prime=(0,v)$. Now $SL_n$ acts on $\A^n\times SL_{n+1}$ by $G\cdot (v,M)=(Gv,GM)$ and we observe that $f$ is $SL_n$-equivariant if $V$ is endowed with the trivial action. It follows that $f$ induces a morphism
\[
f:\A^n\times^{SL_n} SL_{n+1}\to V
\]
where the left-hand term is the quotient of $\A^n\times SL_{n+1}$ by the above action of $SL_n$, and is therefore the vector bundle associated to the torsor $SL_{n+1}\to SL_{n+1}/SL_n$. Moreover, we have a commutative diagram
\[
\xymatrix{\A^n\times^{SL_n} SL_{n+1}\ar[r]^-f\ar[d] & V\ar[d] \\
SL_{n+1}/SL_n\ar[r] & Q_{2n+1}}
\]
where the vertical maps are the projections and the bottom horizontal map is the isomorphism induced by $M\mapsto (M^te_1, M^{-1}e_1)$. It is easy to check that $f$ is an isomorphism locally on $Q_{2n+1}$.
\end{proof}

We now return to our identification of the class in $H^n_{\Nis}(Q_{2n+1},\K^{\MW}_n)$ determined by the morphism $q_{n-1} \circ \delta_{n}$. Recall from \cite[\S 4]{AsokFaselEulerComparison} that the fiber sequence
\[
SL_n/SL_{n-1}\longrightarrow BSL_{n-1}\longrightarrow BSL_{n}
\]
yields a canonical class in $H^n_{\Nis}(BSL_{n},\K^{\MW}_n)$ actually induced by the connecting homomorphism $\bpi_n^{\aone}(BSL_{n})\to \bpi_n^{\aone}(SL_n/SL_{n-1})=\K_n^{\MW}$. The morphism $q_{n-1} \circ \delta_{n}$ is then the pull-back in $H^n_{\Nis}(Q_{2n+1},\K^{\MW}_n)$ of this fundamental class along the morphism $Q_{2n+1}\to BSL_n$. By definition this is the Euler class of $E_n$ considered by Morel in \cite[Theorem 7.14]{MField}.

\begin{rem}
\label{rem:comparisonofEulerclasses}
In the next result, we will implicitly identify two notions of Euler class: the one studied by Morel discussed above, and one studied by the second author in \cite[\S 13]{FaselChowWitt}.  That these two classes coincide can be checked by the method of the universal example; a detailed treatment of this comparison appears in \cite{AsokFaselEulerComparison}.  In brief, the Euler class of a vector bundle can be viewed as being pulled back from a ``universal" Euler class arising from the universal rank $n$ vector bundle over the Grassmannian.  Because both Euler classes under consideration are suitably functorial, it suffices therefore to check that they coincide in the universal case, which can be accomplished, after unwinding the definition of Morel's Euler class as a suitable $k$-invariant in a Postnikov tower, by an explicit computation.
\end{rem}

We now proceed to the computation of $H^{n}_{\Nis}(Q_{2n+1},\K^{\MW}_{n})$. Observe first that the $\aone$-weak equivalence $Q_{2n+1}\simeq SL_{n+1}/SL_{n} \sim {\mathbb A}^{n+1} \setminus 0$ also gives an identification $Q_{2n+1} \cong \Sigma^{n}_s \gm^{\wedge n+1}$.  By means of the suspension isomorphism, the group $H^{n}_{\Nis}(Q_{2n+1},\K^{\MW}_{n})$ is then canonically isomorphic to $H^{0}_{\Nis}(\gm^{\wedge n+1},\K^{\MW}_{n})$.  The group on the right hand side can be described by combining Lemma \ref{lem:cohomology} and Proposition \ref{prop:contractionsofMW}: one obtains an identification $\K^{\MW}_{-1}(k)\cong H^{n}_{\Nis}(Q_{2n+1},\K^{\MW}_{n})$.  By \cite[Lemma 3.10]{MField}, $\K^{\MW}_{-1}(k) \cong W(k)$, and every element of this group is of the form $\eta s$ for $s\in GW(k)$.  Using this notation, we have the following description of the composite map $q_{n-1} \circ \delta_n$.

\begin{lem}
\label{lem:eulerclasscomputation}
We have
\[
q_{n-1}\circ \delta_n=\begin{cases}
\eta & \text{if $n=2m$.} \\
0 & \text{if $n=2m+1$.}\end{cases}
\]
\end{lem}

\begin{proof}
The composite in which we are interested corresponds to the Euler class of $E_n$ by the discussion just preceding Remark \ref{rem:comparisonofEulerclasses}. If $n=2m+1$, then the Euler class of $E_n$ is trivial since a stably free module given by a unimodular row of even length always has a free factor of rank one and thus a trivial Euler class. In case $n=2m$, keeping in mind Remark \ref{rem:comparisonofEulerclasses}, the Euler class is computed in \cite[Proposition 3.2]{Fasel11c}.
\end{proof}

\subsubsection*{The sheaf $\mathbf{S}_n$}
Recall that we have a surjective morphism of sheaves $\K_n^{MW}\to \K_n^M$ for any $n\in\Z$ whose kernel is the principal ideal generated by $\eta$. Using Diagram (\ref{equ:useful}), we get a diagram for $n\geq 3$
\begin{equation}\label{eqn:main}
\xymatrix{\K_{n+1}^{MW}\ar[d]^-{\delta_n} & & &  \\
\bpi_{n-1}(SL_n)\ar[r]_-{q_{n-1}}\ar[d] & \K_n^{MW}\ar[r]_-{\delta_{n-1}}\ar[d] & \bpi_{n-2}(SL_{n-1})\ar[r] & \K_{n-1}^Q\ar[r] & 0 \\
\K_n^Q\ar[d]\ar@{-->}[r] & \K_n^M& & \\
0 & & & }
\end{equation}
Observe moreover that Lemma \ref{lem:eulerclasscomputation} yields an induced morphism $\psi_n:\K_n^Q\to \K_n^M$ for any $n\geq 2$.

\begin{defn}
\label{defn:sn}
For any integer $n\geq 2$, set $\mathbf{S}_n := \operatorname{coker}(\psi_n)$.
\end{defn}

Our next goal will be to show that $\mathbf{S}_n$ is a sheaf of torsion abelian groups and to establish a precise bound on the order of torsion.  To this end, we have to reinterpret the morphism $\psi_n$ slightly.

\subsubsection*{Suslin matrices and a torsion bound}
Let $R$ be a ring. For any integer $m\in\N$ and any row $v=(v_1,\ldots,v_m)\in R^{m}$ we set $v^\prime:=(v_2,\ldots,v_m)\in R^{m-1}$. For any $n\in\N$ and $a=(a_1,\ldots,a_n)$, $b=(b_1,\ldots,b_n)$, we inductively define
matrices $S_n(a,b)$ by $S_1(a_1,b_1)=a_1$ and
\[
S_n(a,b)=\begin{pmatrix} a_1Id_{2^{n-2}} & S_{n-1}(a^\prime,b^\prime) \\ -S_{n-1}(b^\prime,a^\prime)^t & b_1Id_{2^{n-2}}\end{pmatrix}.
\]
If $a$ and $b$ are such that $\sum a_ib_i=1$ then $S_n(a,b)\in GL_{2^{n-1}}(R)$ by \cite[Lemma 5.1]{SuslinStablyFree}. This construction is functorial in $R$ and therefore we obtain a morphism
\[
\phi_n:Q_{2n-1} \longrightarrow GL_{2^{n-1}}.
\]
Composing with the morphism $GL_{2^{n-1}}\to GL$ we obtain a morphism of spaces $Q_{2n-1}\to GL$, which is pointed if we choose $(1,0,\ldots,0,1,0,\ldots,0)$ as base point of $Q_{2n-1}$. Recall that the projection $Q_{2n-1}\to \A^n\setminus 0$ defined by $(x_1,\ldots,x_n,y_1,\ldots,y_n)\to (x_1,\ldots,x_n)$ is an $\A^1$-weak equivalence. As a consequence, $\piaone_{n-1}(Q_{2n-1})\simeq \piaone_{n-1}(\A^n\setminus 0)=\K_n^{MW}$. It follows that $\phi_n$ induces a morphism
\[
\Phi_n:\K_n^{MW} \longrightarrow \K_n^Q.
\]

\begin{lem}\label{lem:sheafifiedmu}
There exists a morphism of strictly $\aone$-invariant sheaves $\mu_n: \K^M_n \to \K^Q_n$ whose sections over finitely generated extensions $L/k$ coincide with the natural homomorphism from Milnor K-theory to Quillen K-theory.  Moreover, there is a canonical identification $(\mu_n)_{-j} = \mu_{n-j}$.
\end{lem}

\begin{proof}
This fact is, more or less, contained in \cite[Remark 5.4]{Rost96}.  More precisely, view $K^M_*(\cdot)$ and $K^Q_*(\cdot)$ as functors on the category of finitely generated extensions $L$ of our base field $k$.  Rost observes \cite[Theorem 1.4, Remark 2.4-2.5]{Rost96} that $K^M_*(\cdot)$ and $K^Q_*(\cdot)$ give rise to cycle modules (he uses the notation $K_*(\cdot)$ for the former and $K'_*(\cdot)$ for the latter). By \cite[Remark 5.4]{Rost96}, the natural homomorphism $K^M_*(\cdot) \to K^Q_*(\cdot)$ yields a morphism of cycle modules.  Now, associated with any cycle module is a family, indexed by the integers, of ``unramified" sheaves (see \cite[Remark 5.2,Corollary 6.5]{Rost96}).  While these ``unramified" sheaves are {\em a priori} Zariski sheaves, one knows that they are in fact already sheaves for the Nisnevich topology \cite[Theorem 8.3.1]{CTHK}, and therefore, the unramified sheaves corresponding to the cycle modules $K^M_*(\cdot)$ and $K^Q_*(\cdot)$ are precisely the sheaves we have been calling $\K^M_i$ and $\K^Q_i$.

The second statement follows by unwinding the definitions of contraction and observing that any cycle module comes equipped with a compatible action of the cycle module $K^M_*(\cdot)$.
\end{proof}

The relation between $\Phi_n$ and $\mu_n$ is summarized in the following result.

\begin{lem}\label{lem:identificationmu}
The diagram
\[
\xymatrix{\K_n^{MW}\ar[r]^-{\Phi_n}\ar[d] & \K_n^Q\ar@{=}[d] \\
\K_n^M\ar[r]_-{\mu_n} & \K_n^Q,}
\]
where the left vertical morphism is the canonical epimorphism $\K^{MW}_n \to \K^M_n$, commutes up to a sign.
\end{lem}

\begin{proof}
For any $n\in\N$, consider the map
\[
\gm \longrightarrow GL_n
\]
defined by $\alpha\mapsto \mathrm{diag}(\alpha,1,\ldots,1)$. Stabilizing, we get a pointed map $\gm\to GL$. Composing with the natural map $GL\to \Omega_s^1BGL$, which is a weak equivalence (see, e.g., \cite[p. 123]{MV}), we obtain a map $\gm\to \Omega_s^1BGL$. By adjunction, this yields a pointed map $\nu_1:\pone\to BGL$. Now $[\pone,BGL]=Pic(\pone)=\Z$ as a consequence of \cite[\S 4 Theorem 3.13]{MV} and it is not hard to see that $\nu_1$ corresponds to $\O(1)$ under this identification.

Using the $H$-space structure on $\Z\times BGL$, we get a map
\[
\xymatrix@C=3.5em{\nu_n:{\pone}^{\wedge n}\ar[r]^-{\nu_1\wedge\ldots\wedge\nu_1} & BGL^{\wedge n} \ar[r] & BGL};
\]
the element of $[(\pone)^{\wedge n},BGL]=\Z$ so obtained corresponds to the line bundle $\O(1)\otimes\ldots\otimes \O(1)$ on ${\mathbb P}^n$ under the natural map $\mathbb P^n\to \mathbb P^n/\mathbb P^{n-1}\simeq (\pone)^{\wedge n}$. It follows that $\nu_n$ is a generator of $[(\pone)^{\wedge n},BGL]=\Z$. On the other hand, $\Phi_n$ is given by the map
\[
\phi_n:Q_{2n-1} \longrightarrow GL
\]
whose class generates $[Q_{2n-1},GL]=\widetilde K_1(Q_{2n-1})=\Z$ by \cite[Theorem 2.3]{Suslin82b}.  Thus, the composite of $\phi_n$ with $GL\to \Omega_s^1BGL$ generates the group $[Q_{2n-1},\Omega_s^1BGL]=[(\pone)^{\wedge n},BGL]$. It coincides therefore with $\nu_n$ up to sign. We are thus left to show that $\nu_n$ induces the map $\K_n^{MW}\to \K_n^m\stackrel{\mu_n}\to \K_n^Q$.

By \cite[Theorems 3.37 and 6.40]{MField} there is, for every integer $n \geq 1$, a canonical ``symbol" morphism of sheaves
\[
\gm^{\wedge n} \longrightarrow \K^{MW}_n \cong \bpi_n^{\aone}({\pone}^{\wedge n});
\]
the induced map on sections over a finitely generated extension $L/k$ assigns the symbol $[a_1,\ldots,a_n]$ in $\K^{MW}_n(L)$ to a section $(a_1,\ldots,a_n)$ of $\gm^{\wedge n}(L)$. The map $\nu_n$ induces a morphism
\[
\bpi_n^{\aone}({\pone}^{\wedge n}) \longrightarrow \bpi_n^{\aone}(BGL) \cong \K^Q_n
\]
and we thus obtain a composite morphism
\[
\gm^{\wedge n} \longrightarrow \K^{MW}_n \longrightarrow \K^Q_n
\]
that assigns to a section $(a_1,\ldots,a_n)$ of $\gm^{\wedge n}(L)$ the cup product $(a_1) \cup \ldots \cup (a_n)$ in $K^Q_n(L)$. The morphism $\K_n^{MW}\to \K_n^Q$ induced by $\nu_n$ then associates the cup product $(a_1) \cup \ldots \cup (a_n)$ in $K^Q_n(L)$ to the symbol $[a_1,\ldots,a_n]\in K_n^{MW}(L)$. Using the relation $[ab]=[a]+[b]+\eta[a,b]$ for any $a,b\in L^\times$ and $(ab)=(a)+(b)$ it is easy to see that this homomorphism factors through $K_n^M(L)$ and the result follows.
\end{proof}

\begin{cor}\label{contraction}
The morphism $(\Phi_n)_{-1}:\K_{n-1}^{MW}\to \K_{n-1}^Q$ coincides with $\Phi_{n-1}$.
\end{cor}

\begin{proof}
Contracting the commutative diagram of the above lemma, we obtain a commutative diagram
\[
\xymatrix{(\K_n^{MW})_{-1}\ar[r]^-{(\Phi_n)_{-1}}\ar[d] & (\K_n^Q)_{-1}\ar@{=}[d] \\
(\K_n^M)_{-1}\ar[r]_-{(\mu_n)_{-1}} & (\K_n^Q)_{-1}.}
\]
By Lemma \ref{lem:contractionsofquillenandmilnor} and Proposition \ref{prop:contractionsofMW}, we see that $(\K_n^{MW})_{-1}=\K_{n-1}^{MW}$, $(\K_n^M)_{-1}=\K_{n-1}^M$ and $(\K_n^Q)_{-1}=\K_{n-1}^Q$. Now the residue maps in the corresponding Gersten complexes commute with either the projection $\K_n^{MW}\to \K_n^M$ or the map $\K_n^M\to \K_n^Q$. This proves the result.
\end{proof}

\begin{lem}\label{lem:commutative}
The following triangle commutes
\[
\xymatrix{\K_n^{MW}\ar[r]^-{\mu_n}\ar[rd]_-{(n-1)!} & \K_n^Q\ar[d]^-{\psi_n} \\
 & \K_n^M}
\]
for any $n\geq 2$.
\end{lem}

\begin{proof}
We first prove that the morphism $\phi_n:Q_{2n-1}\to GL_{2^{n-1}}$ lifts, up to $\aone$-homotopy, to a morphism $\phi^\prime_n:Q_{2n-1}\to GL_n$. Indeed, given $a,b\in R^n$, there exists a matrix $\beta_n(a,b)\in GL_n(R)$ such that its image in $GL_{2^{n-1}}(R)$ is precisely $S_n(a,b)$ up to multiplication by a matrix in $E_n(R)$ (\cite[\S 5b)]{SuslinStablyFree}).  Such a matrix is $\aone$-homotopic to the identity since it can be written as the product of elementary matrices which are evidently $\aone$-homotopic to the identity.

If we write $\Phi'_n$ for the morphism $\K^{MW}_n = \bpi_{n-1}^{\aone}(Q_{2n-1}) \longrightarrow \bpi_{n-1}^{\aone}(GL_n)$ induced by $\phi'_n$, then there is a commutative diagram
\[
\xymatrix{\K_n^{MW}\ar[r]^-{\Phi^\prime_n}\ar[rd]_-{\Phi_n} & \piaone_{n-1}(GL_n)\ar[d] \\
 & \K_n^Q.}
\]

We can now compose the morphism $\phi^\prime_n:Q_{2n-1}\to GL_n$ with the first row map $GL_n\to \A^n\setminus 0$. Using \cite[\S 5]{SuslinStablyFree} once again, we see that this composite maps $(a,b)$ with $a=(a_1,\ldots,a_n)$ and $b=(b_1,\ldots,b_n)$ to $(a_1,a_2,a_3^2,\ldots,a_n^{n-1})$. If $L/k$ is a finitely generated field extension, it follows that the composite
\[
\xymatrix{\K_n^{MW}(L)\ar[r]^-{\Phi^\prime_n} & \piaone_{n-1}(GL_n)(L)\ar[r] & \K_n^{MW}(L)\to \K_n^M(L)}
\]
maps a symbol $[a_1,\ldots,a_n]$ to $\{a_1,a_2,a_3^2,\ldots,a_n^{n-1}\}=(n-1)!\{a_1,\ldots,a_n\}$. Hence the diagram
\begin{equation}
\xymatrix{\K_n^{MW}\ar[r]^-{\Phi_n}\ar[d] & \K_n^Q\ar[d]^-{\psi_n} \\
\K_n^M\ar[r]_-{(n-1)!} & \K_n^M}
\end{equation}
commutes on sections over fields which are finitely generated over $k$. All the sheaves involved are strictly $\A^1$-invariant, and a morphism of strictly $\aone$-invariant sheaves is uniquely determined by its values on sections over finitely generated extensions of the base field.  As a consequence, it follows that the diagram commutes. Using now Lemma \ref{lem:identificationmu} and the fact that $\K_n^{MW}\to \K_n^M$ is surjective we get the required triangle.
\end{proof}

Finally, we can combine the above results to obtain our torsion bound.

\begin{cor}\label{cor:estimate}
The morphism of sheaves $\K_n^M\to \mathbf{S}_n$ induces an epimorphism $\K_n^M/(n-1)!\to \mathbf{S}_n$; for any $m \geq n-2$, this epimorphism becomes an isomorphism after $m$-fold contraction.
\end{cor}

\begin{proof}
The first assertion follows immediately from the previous lemma. After $m$-fold contraction, the triangle of Lemma \ref{lem:commutative} becomes
\[
\xymatrix@C=3em{\K_{n-m}^{M}\ar[r]^-{(\mu_n)_{-m}}\ar[rd]_-{(n-1)!} & \K_{n-m}^Q\ar[d]^-{(\psi_n)_{-m}} \\
 & \K_{n-m}^M}
\]
By Lemma \ref{lem:sheafifiedmu}, we have $(\mu_n)_{-m}=\mu_{n-m}$. If $m\geq n-2$, the multiplication homomorphism $\mu_{n-m}:\K_{n-m}^M\to \K_{n-m}^Q$ is an isomorphism (by Matsumoto's theorem for $n-m=2$) and the result follows.
\end{proof}

\begin{rem}
Observe that, in general, $(\psi_n)_{-m}\neq \psi_{n-m}$. Indeed, let $n=3$ and $m=1$. Using Corollary \ref{contraction}, the above diagram reads:
\[
\xymatrix@C=3em{\K_2^{M}\ar[r]^-{\mu_2}\ar[rd]_-{2} & \K_{2}^Q\ar[d]^-{(\psi_3)_{-1}} \\
 & \K_{2}^M}
\]
and $\mu_2$ is an isomorphism by Matsumoto's theorem. If $L$ is a field such that $K_{2}^M(L)/2$ is non trivial, then $(\psi_3)_{-1}(L)$ is not surjective. On the other hand, $\psi_2$ is an isomorphism.
\end{rem}

\subsubsection*{Factoring the connecting homomorphism}
We now come back to our computation of $\bpi_{n-2}^{\aone}(SL_{n-1})$, and thus to Diagram \ref{eqn:main}:
\[
\xymatrix{\K_{n+1}^{MW}\ar[d]^-{\delta_n} & & &  \\
\bpi_{n-1}(SL_n)\ar[r]_-{q_{n-1}}\ar[d] & \K_n^{MW}\ar[r]_-{\delta_{n-1}}\ar[d] & \bpi_{n-2}(SL_{n-1})\ar[r] & \K_{n-1}^Q\ar[r] & 0 \\
\K_n^Q\ar[d]\ar[r]_{\psi_n} & \K_n^M& & \\
0. & & & }
\]

Suppose first that $n=2m$. In this case, Lemma \ref{lem:eulerclasscomputation} shows that $q_{2m-1}\circ \delta_{2m}=\eta$ and thus the above diagram yields an exact sequence
\[
\xymatrix{0\ar[r] & \mathbf{S}_{2m}\ar[r] & \bpi_{2m-2}(SL_{2m-1})\ar[r] & \K_{2m-1}^Q\ar[r] & 0}
\]
for $m\geq 1$ by definition of the sheaf $\mathbf{S}_{2m}$. If $n=2m+1$, then Lemma \ref{lem:eulerclasscomputation} proves that we have $q_{2m}\circ \delta_{2m+1}=0$. There exists then a morphism $\psi^\prime_{2m+1}$ making the diagram
\[
\xymatrix{\K_{2m+2}^{MW}\ar[d]^-{\delta_{2m+1}} & & &  \\
\bpi_{2m}(SL_{2m+1})\ar[r]_-{q_{2m}}\ar[d] & \K_{2m+1}^{MW}\ar[r]_-{\delta_{2m}}\ar[d] & \bpi_{2m-1}(SL_{2m})\ar[r] & \K_{2m}^Q\ar[r] & 0 \\
\K_{2m+1}^Q\ar[d]\ar@{-->}[ru]_-{\psi^\prime_{2m+1}}\ar[r]_{\psi_{2m+1}} & \K_{2m+1}^M& & \\
0 & & & }
\]
commute and we obtain an exact sequence of the form:
\[
\xymatrix{\K_{2m+1}^Q\ar[r]^-{\psi^\prime_{2m+1}} & \K_{2m+1}^{\MW}\ar[r]^-{\delta_{2m}} & \bpi_{2m-1}(SL_{2m})\ar[r] & \K_{2m}^Q\ar[r] & 0.}
\]

As the sheaf $\K^{MW}_{2m+1}$ is a fiber product of ${\mathbf I}^{2m+1}$ and $\K^M_{2m+1}$ over $\K^M_{2m+1}/2$ by the discussion prior to Proposition \ref{prop:milnorwittexactsequences}, we are left to understand the composite of $\psi^\prime_{2m+1}$ with the surjective morphism $\K^{MW}_{2m+1} \to {\mathbf I}^{2m+1}$.

\begin{lem}\label{lem:image}
The composite morphism $\K^Q_{2m+1} \stackrel{\psi^\prime_{2m+1}}{\to} \K^{MW}_{2m+1} \to {\mathbf I}^{2m+1}$ is trivial.
\end{lem}

\begin{proof}
Consider the diagram
\[
\xymatrix{
& & \bpi_{2m}^{\aone}(SL_{2m+1}) \ar[d] & 0 \ar[d]& & & \\
0 \ar[r]& 2\K^M_{2m+1}\ar[r] & \K^{MW}_{2m+1}\ar[d]\ar[r]\ar[dr]^{\eta} & {\mathbf I}^{2m+1} \ar[r]\ar[d] & 0 & & \\
 & & \bpi_{2m-1}^{\aone}(SL_{2m}) \ar[r]\ar[d] & \K^{MW}_{2m} \ar[r]\ar[d] & \bpi_{2m-2}^{\aone}(SL_{2m-1}) & & \\
& & \K^{Q}_{2m} \ar[r]\ar[d] & \K^M_{2m}\ar[d] & & & \\
& & 0 & 0. & & &
}
\]
The short exact sequence in the second row and the vertical short exact sequence involving ${\mathbf I}^{2m+1}$ are those from Proposition \ref{prop:milnorwittexactsequences}.  Moreover, the commutativity of the triangle with the arrow labeled $\eta$ as its bottom edge is also a consequence of Proposition \ref{prop:milnorwittexactsequences}.  The commutativity of the lower triangle with $\eta$ on the diagonal was established in Lemma \ref{lem:eulerclasscomputation}.

Now, any element in $\bpi_{2m}^{\aone}(SL_{2m+1})$ goes to zero in $\bpi_{2m-1}^{\aone}(SL_{2m})$, and therefore the composite into $\K^{MW}_{2m}$ is also zero.  By commutativity of the diagram, the image of an element in $\bpi_{2m}^{\aone}(SL_{2m+1})$ in ${\mathbf I}^{2m+1}$ is also zero.  Now the morphism $\bpi_{2m}^{\aone}(SL_{2m+1}) \to \K^{MW}_m$ factors through $\psi^\prime_{2m+1}$ by definition, and this yields the first statement.
\end{proof}

As a consequence, we see that $\psi^\prime_{2m+1}$ actually factors as a composite
\[
\xymatrix{\K_{2m+1}^Q\ar[r]^-{\psi_{2m+1}} & 2\K_{2m+1}^M\ar[r] & \K_{2m+1}^{MW}}
\]
where the right arrow is the one of Proposition \ref{prop:milnorwittexactsequences}. It follows that there is morphism of sheaves $\mathbf{S}_{2m+1}\to \K_{2m+1}^M/2$. For any integer $m\in \N$, we define the sheaf $\mathbf{T}_{2m+1}$ as the fiber product
\[
\xymatrix{{\mathbf T}_{2m+1}\ar[r]\ar[d] & {\bf I}^{2m+1}\ar[d] \\
{\mathbf S}_{2m+1}\ar[r] & \K_{2m+1}^M/2.}
\]
In sum, we have established the following result, which establishes Theorem \ref{thmintro:mainhomotopysheaf} stated in the introduction.

\begin{thm}
\label{thm:homotopysheafevencase}
If $k$ is a perfect field (having characteristic unequal to $2$), for any integer $m \geq 1$, there are short exact sequences of strictly $\aone$-invariant sheaves of the form
\[
\begin{split}
0 \longrightarrow {\mathbf T}_{2m+1} \longrightarrow &\bpi_{2m}^{\aone}(BSL_{2m}) \longrightarrow \K^Q_{2m} \longrightarrow 0, \\
0 \longrightarrow {\mathbf S}_{2m+2} \longrightarrow &\bpi_{2m+1}^{\aone}(BSL_{2m+1}) \longrightarrow \K^Q_{2m+1} \longrightarrow 0.
\end{split}
\]
\end{thm}

\subsubsection*{Further computations of contracted homotopy sheaves of $BGL_n$}
We now proceed to describe contractions of the homotopy sheaves of $SL_n$ computed above.  Since $SL_n$ is the $\aone$-connected component of the identity in $GL_n$, the descriptions of homotopy sheaves of degree $\geq 2$ for $GL_n$ are the same as those for $SL_n$.  The computations we now describe will be used later to describe isomorphism classes of vector bundles on spheres, and also to understand compatibility of the computations of homotopy sheaves with corresponding computations in topology.

\begin{prop}
\label{prop:tatehomotopysheavesofgln}
Suppose $nm \geq 1$ and $j \geq 0$ are integers.  There are short exact sequences of the form
\[
\begin{split}
0 \longrightarrow ({\mathbf T}_{2m+1})_{-j} \longrightarrow &\bpi_{2m-1,j}^{\aone}(GL_{2m}) \longrightarrow \K^{Q}_{2m-j} \longrightarrow 0, \text{ and } \\
0 \longrightarrow ({\mathbf S}_{2(m+1)})_{-j} \longrightarrow &\bpi_{2m,j}^{\aone}(GL_{2m+1}) \longrightarrow \K^{Q}_{2m+1-j} \longrightarrow 0.
\end{split}
\]
\end{prop}

\begin{proof}
Combine Lemma \ref{lem:contractionsofquillenandmilnor} with Theorem \ref{thmintro:mainhomotopysheaf}.
\end{proof}


\section{Unimodular rows and vector bundles on split quadrics}
\label{s:unimodular}
In this section, we begin by reviewing some results regarding the relationship between vector bundles and oriented vector bundles, and we then recall some ideas from the theory of unimodular rows in the context of $\aone$-homotopy theory.  We then use the computations of Section \ref{section:metastableSL2n} together with techniques of obstruction theory using the Postnikov tower in $\aone$-homotopy theory (we refer the reader to \cite[\S 2.4]{AsokPi1} for a digest of all the results that will be used) to give a general procedure to describe sets of isomorphism classes of vector bundles.  We refer to vector bundles on the split smooth affine quadric $Q_{2n-1}$ having rank $n-1$ as those of critical rank: above this rank the classification of vector bundles is a stable problem; at or below this rank, the problem is unstable.

\subsubsection*{Oriented vector bundles}
If $X$ is a smooth affine variety, Morel observed \cite[Theorem 8.1.3]{MField} that the set $[X,BGL_r]_{\aone}$ of unpointed $\aone$-homotopy classes of maps is in natural bijection with the set of isomorphism classes of rank $r$ vector bundles on $X$ (when $r = 2$, this relies on results of Moser \cite{Moser}).  As we mentioned in the introduction, it will be technically more convenient to work with maps to $BSL_n$ since the latter is $\aone$-simply connected.  We take as a model for $BSL_n$ the total space of the complement of the zero section of the dual of the determinant of the universal vector bundle over $Gr_n$.  Thus, we view $BSL_n$ as the total space of an explicit $\gm$-torsor over $Gr_n$.  Our present goal is to relate free $\aone$-homotopy classes of maps from a smooth affine scheme to $BSL_n$ to oriented vector bundles.  In what follows, we freely use the terminology of \cite[Appendix A]{MField}.

There is an $\aone$-fiber sequence
\[
BSL_n \longrightarrow Gr_n \longrightarrow B\gm.
\]
Given $X$ a smooth affine $k$-scheme as above, we can map $X_+$ into the above fiber sequence, and the map ``forget the base-point" defines a bijection $[X_+,BSL_n]_{\aone} \isomt [X,BSL_n]_{\aone}$ (pointed $\aone$-homotopy classes on the left and free $\aone$-homotopy classes on the left).  By means of this bijection, we can work with free $\aone$-homotopy classes just as we would work with pointed $\aone$-homotopy classes.  Since $BSL_n$ is the $\aone$-homotopy fiber of the map $Gr_n \to B\gm$, we know that lifts $X \to BSL_n$ of a given map $X \to Gr_n$ correspond bijectively
with chosen null-$\aone$-homotopies of the composite map $X \to B\gm$.

The map $Gr_n \to B\gm$ is induced by the determinant, and given a vector bundle $\mathcal{E}$ on $X$ corresponding to a map $X \to Gr_n$, the composite map $X \to B\gm$ sends $\mathcal{E}$ to the line bundle $\det \mathcal{E}$.  The constant map $X_+ \to B\gm$ sending everything to the base-point corresponds to the trivial bundle on $X$, so the composite map $X \to B\gm$ is null-$\aone$-homotopic if and only if $\det \mathcal{E}$ is trivial.  A choice of null-$\aone$-homotopy therefore yields a choice of isomorphism $\det \mathcal{E} \cong \O_X$.  Thus, we obtain a function
\[
[X,BSL_n]_{\aone} \longrightarrow \mathscr{V}^o_n(X).
\]
To see this map is a bijection one can repeat, essentially verbatim, the steps of Morel's identification of maps to the Grassmannian with isomorphism classes of vector bundles replacing $Gr_n$ with the model of $BSL_n$ mentioned above.  The first necessary technical result is the following one.

\begin{thm}[Morel]
\label{thm:affinezariskibg}
The space $Sing_*^{\aone}(BSL_n)$ satisfies the affine BG property in the Zariski topology.
\end{thm}

\begin{proof}
This is proven in exactly the same fashion as \cite[Theorem 9.19 and Corollary 9.20]{MField}.  One need only observe that the results referenced hold not just for projective modules, but also projective modules equipped with a chosen trivialization of the determinant.
\end{proof}

\begin{thm}
\label{thm:orientedvectorbundles}
If $X$ is a smooth affine scheme, there is a bijection
\[
[X,BSL_n]_{\aone} \cong \mathscr{V}^o(X).
\]
\end{thm}

\begin{proof}
We know that $Sing_*^{\aone}(SL_r)$ satisfies the affine BG property in the Nisnevich topology by the results of Morel and Moser \cite[Theorem 9.21]{MField} or \cite{Moser}.  By the discussion above, $Sing_*^{\aone}(BSL_n)$ satisfies the affine BG property in the Zariski topology.  The result then follows from \cite[Theorem A.19]{MField} since $\Omega^1_s Sing_*^{\aone}(BSL_n)^f \cong Sing_*^{\aone}(SL_n)$.
\end{proof}

\subsubsection*{Unimodular rows and naive $\aone$-homotopies}
Let $R$ be a ring and let $n\geq 3$ be an integer. Recall that a row $(a_1,\ldots,a_n)$ of elements of $R$ is called unimodular if there exists $(b_1,\ldots,b_n)$ such that $\sum a_ib_i=1$.  We write $Um_n(R)$ for the set of unimodular rows of length $n$ over $R$.  The set $Um_n(R)$ is naturally pointed by the row $e_1 := (1,0,\ldots,0)$.  The group $GL_n(R)$ acts on $Um_n(R)$ by multiplication on the right and therefore we can consider the induced action of any subgroups of $GL_n(R)$ on $Um_n(R)$; in this paper, we will mainly be interested in the subgroup $SL_n(R)$ and the subgroup $E_n(R)$ generated by elementary matrices.

Given two smooth $k$-schemes $X$ and $Y$ and a pair of morphisms $f,g: X \to Y$, we say that $f$ and $g$ are naively $\aone$-homotopic if there exists a morphism $H: X \times \aone \to Y$ such that $H(0) = f$ and $H(1) = g$.  Naive $\aone$-homotopy is in general not an equivalence relation and we write $\sim_{\aone}$ for equivalence relation it generates: we set
\[
\hom_{\aone}(X,Y) := \hom_{\Sm_k}(X,Y)/ \sim_{\aone}.
\]
We now specialize to the case $Y = {\mathbb A}^n \setminus 0$.

If $k$ is a field, and $R$ is a (commutative, unital) $k$-algebra, a unimodular row $(a_1,\ldots,a_n)$ can be viewed as a morphism $\Spec R \to {\mathbb A}^n \setminus 0$.  This identification yields a bijection
\[
Um_n(R) \isomto \hom_{\Sm_k}(\Spec R,{\mathbb A}^n \setminus 0).
\]
In this context, two unimodular rows that can be related by right multiplication by an elementary matrix are naively $\aone$-homotopic, and provided $R$ is a smooth $k$-algebra one can show that the induced map
\[
Um_n(R)/E_n(R) \longrightarrow \hom_{\aone}(\Spec R,{\mathbb A}^n \setminus 0)
\]
is a (pointed) bijection \cite[Theorem 2.1]{Fasel08b}.

\subsubsection*{Unimodular rows and $\aone$-homotopy classes}
For any pair of smooth $k$-schemes $X$ and $Y$, the map $\hom_{\Sm_k}(X,Y) \to [X,Y]_{\aone}$ factors through a map
\[
\hom_{\aone}(X,Y) \longrightarrow [X,Y]_{\aone}
\]
since naively $\aone$-homotopic morphisms become equal in $\ho{k}$.  In the special case where $X$ is smooth affine scheme and $Y=\A^n\setminus 0$, the map $\hom_{\aone}(X,\A^n\setminus 0)\to [X,\A^n\setminus 0]_{\aone}$ is in fact a bijection \cite[Remark 8.10]{MField}. It follows that the right-hand side is generated by morphisms of schemes $X\to \A^n\setminus 0$, i.e., unimodular rows of length $n$ over $\O_X(X)$.

The set $[X,{\mathbb A}^n \setminus 0]_{\aone}$ admits an alternate cohomological description, at least provided $n \geq 3$.  Indeed, in that case, ${\mathbb A}^n \setminus 0$ is $\aone$-simply connected.  If $X(k)$ is non-empty, then Lemma \ref{lem:basepoint} shows that the canonical map from pointed $\aone$-homotopy classes to free $\aone$-homotopy classes is a bijection.

Thus, assume we have a pointed $\aone$-homotopy class of maps from $X$ to ${\mathbb A}^n\setminus 0$.  We now consider the $\aone$-Postnikov tower of ${\mathbb A}^n \setminus 0$ (see \cite[\S 2.4]{AsokPi1} for recollections).  Since $\bpi_{n-1}^{\aone}({\mathbb A}^n \setminus 0) = \K^{MW}_n$, and ${\mathbb A}^n \setminus 0$ is $\aone$-$(n-2)$-connected,  we know that the first non-trivial piece of the $\aone$-Postnikov tower is $({\mathbb A}^n \setminus 0)^{(n-1)}$, and that this space is an Eilenberg-Mac Lane space of the form $K(\K^{MW}_n,n-1)$.  Furthermore, there is an $\aone$-fiber sequence
\[
({\mathbb A}^n \setminus 0)^{(n)} \longrightarrow K(\K^{MW}_n,n-1) \longrightarrow K(\bpi_n^{\aone}({\mathbb A}^n \setminus 0),n+1).
\]
Moreover, there are canonical morphisms ${\mathbb A}^n \setminus 0 \longrightarrow ({\mathbb A}^n \setminus 0)^{(i)}$ for arbitrary $i$.

If $(X,x)$ is a pointed smooth scheme, then there is an induced map
\[
[X,{\mathbb A}^n \setminus 0]_{\aone} \longrightarrow H^{n-1}(X,\K^{MW}_n)
\]
obtained by composing with the morphism ${\mathbb A}^n \setminus 0 \to ({\mathbb A}^n \setminus 0)^{(n-1)}=K(\K_n^{MW},n-1)$.

\begin{prop}[{\cite[Theorem 7.16 footnote 11]{MField}}]
\label{prop:degreehomomorphism}
If $X$ is a smooth scheme having the $\aone$-homotopy type of a smooth scheme of dimension $\leq n$ (resp. $\leq n-1)$, then the map
\[
[X,{\mathbb A}^n \setminus 0]_{\aone} \longrightarrow H^{n-1}(X,\K^{MW}_n)
\]
is surjective (resp. bijective).
\end{prop}

\begin{proof}
Since $X$ has the $\aone$-homotopy type of a smooth scheme of dimension $\leq n$, it follows that the map
\[
[X,{\mathbb A}^n \setminus 0]_{\aone} \longrightarrow [X,({\mathbb A}^n \setminus 0)^{(n)}]_{\aone}
\]
is a bijection because all further obstructions to lifting live in groups that vanish, and all subsequent lifts are uniquely determined.

Because of the $\aone$-fiber sequences
\[
({\mathbb A}^n \setminus 0)^{(n)} \longrightarrow K(\K^{MW}_n,n-1) \longrightarrow K(\bpi_n^{\aone}({\mathbb A}^n \setminus 0),n+1),
\]
a map $X \to K(\K^{MW}_n,n-1)$ extends to $({\mathbb A}^n \setminus 0)^{(n)}$ if and only if the composite to $K(\bpi_n^{\aone}({\mathbb A}^n \setminus 0),n+1)$, i.e., an element of $H^{n+1}(X,\bpi_n^{\aone}({\mathbb A}^n \setminus 0))$ is trivial.  This group vanishes by the hypothesis on the dimension (and hence Nisnevich cohomological dimension) of $X$.  Given that we know a lift exists, the space of lifts is parameterized by $H^n(X,\bpi_n^{\aone}({\mathbb A}^n \setminus 0))$.  If $X$ has the $\aone$-homotopy type of a smooth scheme of dimension $\leq n-1$, then this group also vanishes, which gives the bijectivity statement.
\end{proof}

Proposition \ref{prop:degreehomomorphism} can be strengthened slightly when $X = {\mathbb A}^n \setminus 0$ using a cohomological vanishing result (Lemma \ref{lem:cohomology}) we will establish in the next subsection.

\begin{cor}
\label{cor:GWdegree}
When $X = {\mathbb A}^n \setminus 0$, the surjection
\[
[{\mathbb A}^n \setminus 0,{\mathbb A}^n \setminus 0]_{\aone} \longrightarrow H^{n-1}({\mathbb A}^n \setminus 0,\K^{MW}_n)
\]
is actually a bijection.
\end{cor}


\subsubsection*{A refined vanishing statement}
The following technical lemma will be useful in describing oriented vector bundles.

\begin{lem}
\label{lem:cohomology}
If $\bf A$ is a strictly $\aone$-invariant sheaf and if $n\geq 2$ is an integer, then
\[
H^i(Q_{2n-1},\bf A)\simeq  \begin{cases} {\bf A}(k) & \text{ if } i=0, \\ {\bf A}_{-n}(k) & \text{ if } i=n-1, \text{ and } \\ 0 & \text{ otherwise. } \end{cases}
\]
\end{lem}

\begin{proof}
Since $p_{2n-1}$ is an isomorphism in $\ho k$ and the Eilenberg-Mac Lane space $K({\mathbf A},i)$ is $\aone$-local for every integer $i \geq 0$, we have $H^i(Q_{2n-1},\mathbf{A})=[Q_{2n-1},K(\mathbf {A},i)]_{\aone}=[\mathbb A^n\setminus 0,K(\mathbf {A},i)]_{\aone}$. If $i=0$, we have $H^0(\mathbb A^n\setminus 0,\mathbf{A})=\mathbf{A}(k)$ since $0\in \mathbb A^n$ is of codimension $\geq 2$. If $i\geq 1$, we have
\[
[\mathbb A^n\setminus 0,K(\mathbf {A},i)]_{\aone}=[\Sigma^{n-1}\gm^{\wedge n},K(\mathbf {A},i)]_{\aone}=\bpi_{n-1,n}^{\aone}(K(\mathbf {A},i))(k)=\bpi_{n-1}^{\aone}(K(\mathbf {A},i))_{-n}(k)
\]
The latter is trivial if $i\neq n-1$ and equal to $\mathbf{A}_{-n}(k)$ if $i=n-1$.
\end{proof}

\subsubsection*{Completability of unimodular rows}
Consider the morphism $SL_n \to {\mathbb A}^n \setminus 0$ given by ``projection onto the first column."  If $X$ is a smooth affine scheme, then a unimodular row of length $n$ on $X$ is completable if there exists a morphism $X \to SL_n$ factoring the morphism $X \to {\mathbb A}^n \setminus 0$ that determines the row.  We saw above that the morphism $\hom_{\Sm_k}(X,{\mathbb A}^n \setminus 0) \to [X,{\mathbb A}^n \setminus 0]_{\aone}$ is surjective.  Likewise, the canonical morphism $\hom_{\Sm_k}(X,SL_n) \to [X,SL_n]_{\aone}$ is surjective because $Sing_*^{\aone}(SL_n)$ satisfies the so-called affine BG property \cite[Theorem 9.21]{MField}, and we have a commutative diagram of the form
\[
\xymatrix{
\hom_{\Sm_k}(X,SL_n) \ar[r]\ar[d] & \hom_{\Sm_k}(X,{\mathbb A}^n \setminus 0) \ar[d]\\
[X,SL_n]_{\aone} \ar[r]& [X,{\mathbb A}^n \setminus 0]_{\aone},
}
\]
where the horizontal morphisms are induced by the ``projection onto the first column."  The vertical morphisms are surjective because they are quotient morphisms corresponding to taking $E_n(\O_X(X))$-orbits.  Therefore to see if a unimodular row is completable, it suffices to show that the corresponding $\aone$-homotopy class lifts to $[X,SL_n]_{\aone}$.

To solve this lifting problem, observe that there is an $\aone$-fiber sequence of the form
\[
SL_n \longrightarrow {\mathbb A}^n \setminus 0 \longrightarrow BSL_{n-1} \longrightarrow BSL_n.
\]
Now Lemma \ref{lem:universaltorsor} together with the $\aone$-weak-equivalence $SL_n/SL_{n-1}\simeq \A^n\setminus 0$ shows that the map ${\mathbb A}^n \setminus 0 \longrightarrow BSL_{n-1}$ classifies the oriented vector bundle on $\A^n\setminus 0$ given by the kernel of the surjective homomorphism $(x_1,\ldots,x_n):\O_{\A^n\setminus 0}^n\to \O_{\A^n\setminus 0}$.
If we map $X_+$ into this $\aone$-fiber sequence, we obtain an exact sequence of the form
\[
\cdots \longrightarrow [X,SL_n]_{\aone} \longrightarrow [X,{\mathbb A}^n \setminus 0]_{\aone} \longrightarrow [X,BSL_{n-1}]_{\aone} \longrightarrow [X,BSL_n];
\]
here we consider free homotopy classes of maps, and the two pointed sets on the right are pointed by the trivial bundle of the appropriate rank on $X$ equipped with the identity trivialization of the determinant.

The map $[X,{\mathbb A}^n \setminus 0]_{\aone} \to [X,BSL_{n-1}]_{\aone}$ is obtained by sending a map $X \to {\mathbb A}^n \setminus 0$ to the pullback of the vector bundle described above.  Thus, we need information about the sets of free $\aone$-homotopy classes of maps $[X,BSL_{n-1}]_{\aone}$.  Since $BSL_{n-1}$ is $\aone$-simply connected, we can apply Lemma \ref{lem:basepoint} to deduce the following fact.

\begin{lem}
\label{lem:orientedvectorbundles}
If $(X,x)$ is a pointed smooth affine scheme, then for any integer $r$, there is a canonical bijection
\[
[(X,x),(BSL_r,\ast)]_{\aone} \isomto [X,BSL_r]_{\aone}.
\]
In particular, when $X = Q_{2n-1}$, we see that $\bpi_{n-1,n}^{\aone}(BSL_{n-1})(k) \isomt [Q_{2n-1},BSL_r]_{\aone}$.
\end{lem}






\subsubsection*{Vector bundles of large rank}
\begin{cor}\label{cor:mgeqn}
If $n \geq 1$ is any integer, any vector bundle $E$ of rank $m\geq n$ over $Q_{2n-1}$ is free.
\end{cor}

\begin{proof}
If $n = 1$, then $Q_{2n-1} \cong \gm$ and the result is clear since all vector bundles on $\gm$ are free.  Therefore, assume $n \geq 2$.  Then, the $\aone$-weak equivalence $Q_{2n-1} \to {\mathbb A}^n \setminus 0$ shows that $Pic(Q_{2n-1})$ is trivial.  Therefore, any vector bundle on $Q_{2n-1}$ has trivial determinant, and we can fix arbitrarily a trivialization of the determinant to get an element of $[Q_{2n-1},BSL_m]_{\aone}$. We are thus left to prove that this set is trivial. As we observed above in Lemma \ref{lem:orientedvectorbundles}, the canonical map
\[
\bpi_{n-1,n}^{\aone}(BSL_m)(k) \longrightarrow [Q_{2n-1},BSL_m]_{\aone}
\]
is a bijection.

By \cite[Theorem 6.13]{MField}, we have identifications
\[
\bpi_{n-1,n}^{\aone}(BSL_m)(k) \cong \bpi_{n-1}^{\aone}(BSL_m)_{-n}(k).
\]
If $m \geq n$, then $\bpi_{n-1}^{\aone}(BSL_m)_{-n} = (\K^Q_{n-1})_{-n}$ by the stabilization results of Morel (see Theorem \ref{thm:stabilizationsln}).  Since $(\K^Q_{n-1})_{-n} = 0$, it follows that $[Q_{2n-1},BSL_m]_{\aone}=*$.  Consequently, the set of vector bundles of rank $m$ also consists of a single element.
\end{proof}

\subsubsection*{Vector bundles of critical rank I: the case $n$ even}
We now study vector bundles of rank $n-1$ on $Q_{2n-1}$ under the additional assumption that $n$ is even; we begin by describing oriented vector bundles.

\begin{thm}\label{thm:bundleseven}
If $n\geq 2$ is an even integer, then there is a bijection between the set of isomorphism classes of oriented rank $(n-1)$ vector bundles on $Q_{2n-1}$ and the group $\Z/(n-1)!$.  Moreover, each isomorphism class admits a representative given by the unimodular row $(x_1^m,x_2,\ldots,x_n)$ for $1\leq m\leq (n-1)!$.
\end{thm}

\begin{proof}
By Lemma \ref{lem:orientedvectorbundles} there is a bijection
\[
\bpi_{n-1,n}^{\aone}(BSL_{n-1})(k) \isomto [Q_{2n-1},BSL_{n-1}]_{\aone}.
\]
The former sheaf is described in Proposition \ref{prop:tatehomotopysheavesofgln}.

Since $n$ is even, $n-1$ is odd, and we have a short exact sequence of the form
\[
0 \longrightarrow ({\mathbf S}_n)_{-n}(k) \longrightarrow \bpi_{n-1}^{\aone}(BSL_{n-1})_{-n}(k) \longrightarrow (\K^Q_{n-1})_{-n}(k) \longrightarrow 0.
\]
Since $(\K^Q_{n-1})_{-n} = 0$, it follows that there is a bijection $[Q_{2n-1},BSL_{n-1}]_{\aone} \isomt ({\mathbf S}_n)_{-n}(k)$.  Corollary \ref{cor:estimate} yields an identification $({\bf S}_n)_{-n}(k)=\Z/(n-1)!$.

Next, we give explicit descriptions of all the non-isomorphic rank $n-1$ vector bundles whose existence is guaranteed by the preceding discussion.  To this end, we consider the $\aone$-fiber sequence
\[
\A^n\setminus 0 \longrightarrow BSL_{n-1} \longrightarrow BSL_n.
\]
Mapping $(Q_{2n-1})_+$ into this fiber sequence yields an exact sequence of (groups and) pointed sets
\[
[Q_{2n-1},\A^n\setminus 0]_{\aone} \longrightarrow [Q_{2n-1},BSL_{n-1}]_{\aone} \longrightarrow [Q_{2n-1},BSL_n]_{\aone}
\]
where the first map is obtained by pulling-back the universal rank $n-1$ stable free bundle on $\A^n\setminus 0$.

Corollary \ref{cor:mgeqn} states that $[Q_{2n-1},BSL_n]_{\aone}=*$ so therefore any of the vector bundles just considered becomes trivial upon adding a free rank $1$ summand.  In other words, all such vector bundles are given by unimodular rows.  By the previous paragraphs, we know that $[Q_{2n-1},BSL_{n-1}]_{\aone}\simeq\bpi_{n-1,n}^{\aone}(BSL_{n-1})(k) \cong ({\mathbf S}_n)_{-n}(k)$. Moreover, $[Q_{2n-1},\A^n\setminus 0]_{\aone}\simeq\bpi_{n-1,n}^{\aone}(\A^n\setminus 0)(k) \cong K_0^{\MW}(k)$. Under these identifications, it is clear that the morphism
\[
[Q_{2n-1},\A^n\setminus 0]_{\aone} \longrightarrow [Q_{2n-1},BSL_{n-1}]_{\aone}
\]
corresponds to the projection $K_0^{\MW}(k)\to K_0^M(k)\to ({\mathbf S}_n)_{-n}(k)=\Z/(n-1)!$. Consider next the class of the unimodular row $(x_1^m,x_2,\ldots,x_n)$ in $[Q_{2n-1},\A^n\setminus 0]_{\aone}$ for any $m\in\N$. Under the weak-equivalence $Q_{2n-1}\simeq \A^n\setminus 0\simeq \Sigma^{n-1}\gm^{\wedge n}$, this class corresponds to the morphism $l:\gm^{\wedge n}\to \gm^{\wedge n}$ defined by $l(x_1,\ldots,x_n)=(x_1^m,x_2,\ldots,x_n)$. As $K_n^M(\gm^{\wedge n})=\Z$ is generated by the class of the symbol $\{x_1,\ldots,x_n\}$ and $\{x_1^m,x_2,\ldots,x_n\}=m\{x_1,\ldots,x_n\}$, it follows that the unimodular rows $(x_1^m,x_2,\ldots,x_m)$ for $1\leq m\leq (n-1)!$ give all the isomorphism classes of oriented vector bundles of rank $n-1$ on $Q_{2n-1}$.
\end{proof}

\begin{cor}\label{cor:nonorientedbundleseven}
If $n\geq 2$ is an even integer, then the set of isomorphism classes of rank $n-1$ vector bundles on $Q_{2n-1}$ has $(n-1)!$ elements.  Moreover, each isomorphism class admits a representative given by the unimodular row $(x_1^m,x_2,\ldots,x_n)$ for $1\leq m\leq (n-1)!$.
\end{cor}

\begin{proof}
By Corollary \ref{cor:mgeqn}, we know that any vector bundle of rank $n$ on $Q_{2n-1}$ is free. It follows that the set of isomorphism classes of vector bundles of rank $n-1$ is given by the orbit set $Um_n(Q_{2n-1})/GL_n(Q_{2n-1})$. On the other hand, Theorem \ref{thm:bundleseven} yields $Um(Q_{2n-1})/SL_n(Q_{2n-1})=\Z/(n-1)!$ with explicit generators. To prove the result, it suffices thus to show that the surjective map $Um_n(Q_{2n-1})/SL_n(Q_{2n-1})\to Um_n(Q_{2n-1})/GL_n(Q_{2n-1})$ is injective. This amounts to show that for any unit $\alpha\in A_{2n-1}^\times=k^\times$ and any $m\in\N$ the classes of the unimodular rows $(x_1^m,x_2,\ldots,x_n)$ and $(\alpha x_1^m,x_2,\ldots,x_n)$ coincide in $Um_n(Q_{2n-1})/SL_n(Q_{2n-1})$. This is clear since we have
\[
\{\alpha x_1^m,x_2,\ldots,x_n\}=\{\alpha,x_2,\ldots,x_n\}+\{x_1^m,x_2,\ldots,x_n\}=\{x_1^m,x_2,\ldots,x_n\}
\]
in $K_n^M(\gm^{\wedge n})$.
\end{proof}

\subsubsection*{Vector bundles of critical rank II: the case $n$ odd}
We now study isomorphism classes of oriented rank $n-1$ vector bundles on $Q_{2n-1}$ when $n$ is odd.  In that case, recall from section \ref{section:metastableSL2n} that there is an exact sequence of the form
\[
\xymatrix{0\ar[r] & {\bf T}_{n}\ar[r] & \piaone_{n-1}(BSL_{n-1})\ar[r] & \K_{n-1}^Q\ar[r] & 0}
\]
where ${\bf T}_{n}$ is the image of the sheaf $\K_{n}^{MW}=\piaone_{n-1}(\A^{n}\setminus 0)$ in $ \piaone_{n-1}(BSL_{n-1})$ under the morphism of sheaves induced by the morphism of spaces $\A^{n}\setminus 0\to BSL_{n-1}$.

\begin{thm}
\label{thm:bundlesodd}
If $n\geq 3$ is an odd integer, then there is a bijection between the set of isomorphism classes of oriented rank $n-1$ vector bundles on $Q_{2n-1}$ and $\Z/(n-1)!\times_{\Z/2} W(k)$.  Moreover, the set of unimodular rows $\{(\alpha x_1,x_2,\ldots,x_n)\vert \alpha\in k^\times\}$ together with the rows $(x_1^{2m},x_2,\ldots,x_n)$ for $1\leq m\leq (n-1)!/2$ give representatives of the isomorphism classes of oriented vector bundles on $Q_{2n-1}$.
\end{thm}

\begin{proof}
Repeating the beginning of the proof of Theorem \ref{thm:bundleseven}, we observe that the set $[Q_{2n-1},BSL_{n-1}]_{\aone}$ can be identified with $\bpi_{n-1}^{\aone}(BSL_{n-1})_{-n}(k)$ and that all vector bundles of rank $n-1$ become trivial upon addition of a free rank $1$ summand and are given by unimodular rows.  In this case, since $n$ is odd, $n-1$ is even and then Proposition \ref{prop:tatehomotopysheavesofgln} implies that there is an isomorphism of the form
\[
\bpi_{n-1}^{\aone}(BSL_{n-1})_{-n}(k) \cong ({\mathbf T}_{n})_{-n}(k).
\]
Moreover, the fiber product presentation of ${\mathbf T}_n$ yields a presentation of the form
\[
\xymatrix{({\bf T}_n)_{-n}(k)\ar[r]\ar[d] & ({\bf I}^n)_{-n}(k)\ar[d] \\
\Z/(n-1)!\ar[r] & \Z/2}
\]
By Proposition \ref{prop:contractionsofMW}, we know that $({\bf I}^n)_{-n}={\bf W}$ and therefore that the set of oriented rank $n-1$ vector bundles is in bijection with
\[
\Z/(n-1)! \times_{\Z/2} W(k).
\]
This presentation as a fiber product gives rise to an exact sequence of the form
\[
0 \longrightarrow 2\Z/(n-1)! \longrightarrow \Z/(n-1)! \times_{\Z/2} W(k) \longrightarrow W(k) \longrightarrow 0,
\]
where the first term should be thought of as multiples of the hyperbolic form $\langle 1,-1\rangle$.

To find explicit generators, we observe as in the proof of Theorem \ref{thm:bundleseven} that the morphism
\[
[Q_{2n-1},\A^n\setminus 0]_{\aone} \longrightarrow [Q_{2n-1},BSL_{n-1}]_{\aone}
\]
corresponds to the projection $\Z\times_{\Z/2}W(k)=K_0^{\MW}(k)\to ({\mathbf T}_n)_{-n}(k)=\Z/(n-1)!\times_{\Z/2} W(k)$. Since the class of a unimodular row given by an endomorphism of ${\mathbb A}^n\setminus 0$ in the group on the left-hand side of the above morphism is precisely the Grothendieck-Witt valued Brouwer degree, it suffices to compute this degree for the rows in the theorem statement.

That $(\alpha x_1,\ldots,x_n)$ has degree precisely $(1,\langle \alpha \rangle)$ is a consequence of \cite[Remark 2.6]{Fasel11c}. On the other hand, the degree of the row $(x_1^{2m},x_2,\ldots,x_n)$ is given by the class of the symbol $[x_1^{2m},x_2,\ldots,x_n]$ as a $K_0^{MW}(k)$-multiple of the class of $[x_1,\ldots,x_n]$. Now \cite[Lemma 3.14]{MField} yields $[x_1^{2m},x_2,\ldots,x_n]=m\langle 1,-1\rangle [x_1,\ldots,x_n]$ showing that $(x_1^{2m},x_2,\ldots,x_n)$ has degree $(2m,0)$.
\end{proof}

\begin{rem}
Using \cite[Lemma 3.14]{MField} once again, we see that the degree of $(x_1^m,x_2,\ldots,x_n)$ is $(m,\langle 1,-1,\ldots,(-1)^{m-1}\rangle)$. For the sake of completeness, we also show how to compute the degree of $(\alpha x_1^m,x_2,\ldots,x_n)$ for any $\alpha\in k^\times$ and any $m\in\N$. We have
\[
[\alpha x_1^m]=[\alpha]+[x_1^m]+\eta[\alpha,x_1^m]=[\alpha]+\langle \alpha\rangle [x_1^m]=[\alpha]+\langle \alpha,-\alpha,\ldots,(-1)^m\alpha\rangle[x_1].
\]
Since the unimodular row $(\alpha,x_2,\ldots,x_n)$ is completable, it follows that the degree of the row $(\alpha x_1^m,x_2,\ldots,x_n)$ is $(m,\langle \alpha,-\alpha,\ldots,(-1)^m\alpha\rangle)$.
\end{rem}

\begin{cor}\label{cor:nonorientedbundlesodd}
If $n\geq 3$ is an odd integer, then there is a bijection between the set of isomorphism classes of rank $n-1$ vector bundles on $Q_{2n-1}$ and the quotient set $\Z/(n-1)!\times_{\Z/2} W(k)/k^{\times}$.
\end{cor}

\begin{proof}
Following the arguments of Corollary \ref{cor:nonorientedbundleseven}, we see that we have to compare the degrees in $\Z/(n-1)!\times_{\Z/2} W(k)/k^{\times}$ of $(a_1,\ldots,a_n)$ and $(\alpha a_1,a_2,\ldots,a_n)$ for $\alpha\in Q_{2n-1}^\times=k^\times$. Arguing as in the above remark, we see that a unit $\alpha\in k^\times$ acts on a pair $(m,\phi)\in \Z/(n-1)!\times_{\Z/2} W(k)$ as $\alpha\cdot (m,\phi)=(m,\langle \alpha\rangle\cdot \phi)$. The result follows.
\end{proof}

\begin{rem}
Let $f_1,\ldots,f_n \in k[y_1,\ldots,y_n]$ be functions such that $V(f_1,\ldots,f_n)$ is a point in ${\mathbb A}^n$.  The hypersurface defined by the equation $\sum_i x_i f_i = 1$ is a smooth affine variety that is $\aone$-weakly equivalent to ${\mathbb A}^n \setminus 0$.  By Morel's theorem, the set of isomorphism classes of vector bundles on such a variety is canonically in bijection with the set of isomorphism classes of vector bundles on $Q_{2n-1}$.  However, the varieties so defined are {\em not} in general isomorphic to $Q_{2n-1}$.  These varieties are torsors under vector bundles over ${\mathbb A}^n \setminus 0$.  For example, when $n = 2$, there are pairwise non-isomorphic varieties of this form \cite[Theorem 2.5]{DuboulozFinston}.  Theorems \ref{thm:bundleseven} and \ref{thm:bundlesodd} also provide a description of the set of isomorphism classes of rank $n-1$ bundles on any such variety.
\end{rem}

\subsubsection*{Vector bundles below critical rank}
Since the Picard group of $Q_7$ is trivial, and we understand vector bundles of rank $\geq 3$ on $Q_7$ by the results already proven, the next result, which uses \cite[Theorem 3.10]{Asok12} completes the description of vector bundles on $Q_7$; this uses \cite[Proposition 7.1]{Asok12} in the same way as the results above, so we leave the proof to the reader.

\begin{prop}
\label{prop:rank2bundlesonq7}
There is a canonical bijection $\mathscr{V}_2(Q_7) \isomt \bpi_{2}^{\aone}(SL_2)_{-4}(k)$ and a short exact sequence of the form
\[
0 \longrightarrow I(k) \longrightarrow \bpi_{2}^{\aone}(SL_2)_{-4} \longrightarrow \Z/12 \longrightarrow 0.
\]
\end{prop}

\begin{rem}
As with the case of rank $3$ bundles on $Q_7$, the rank $2$ vector bundles on $Q_7$ are all given by stably free modules. It is possible to give explicit representatives for each of these stably free vector bundles: see \cite[\S 3]{Fasel10} for more information on how to associate a symplectic bundle of rank $2$ to an unimodular row of length $4$.  For example, the unimodular rows of the form $(x_1^m,x_2,x_3,x_4)$ with $1 \leq m \leq 12$ give rise to non-isomorphic rank $2$ vector bundles.
\end{rem}

\section{Two applications: completability and complex realization}
\label{section:applications}
In this section, we discuss two applications of the description of the set of isomorphism classes of vector bundles on split quadrics from Section \ref{s:unimodular}.

\subsection*{A question of M. V. Nori and lifting problems}
Our computations regarding vector bundles of rank $n-1$ on $Q_{2n-1}$ allow us to address the following question of M. V. Nori on unimodular rows.

\begin{question}[M. V. Nori]
\label{quest:nori}
Suppose $k$ is a field, $R=k[x_1,\ldots,x_n]$ is a polynomial ring in $n$ variables over $k$, $\phi:R\to A$ is a $k$-algebra homomorphism such that $\sum \phi(x_i)A=A$, and $f_1,\ldots,f_n$ are elements of $R$ such that the reduced closed subscheme defined by the ideal $(f_1,\ldots,f_n)$ is $0 \in {\mathbb A}^n$.  If $length(R/(f_1,\ldots,f_n))$ is divisible by $(n-1)!$, then is the unimodular row $(\phi(f_1),\ldots,\phi(f_n))$ completable?
\end{question}

Nori's question admits the following reinterpretation.  The homomorphism $\phi:R\to A$ such that $\sum \phi(x_i)A=A$ defines a unimodular row $v=(\phi(x_0),\ldots,\phi(x_n))$ and a morphism of schemes $v:\Spec R\to \A^{n}\setminus 0$. Now  any polynomials $f_1,\ldots,f_n$ such that $\mathrm{rad}(f_1,\ldots,f_n)=(x_1,\ldots,x_n)$ defines a morphism $\varphi:\A^n\setminus 0\to \A^n\setminus 0$. If $l(R/(f_1,\ldots,f_n))$ is divisible by $(n-1)!$, then does the morphism $\varphi\circ v:\Spec A\to \A^n\setminus 0$ lift to a morphism $\Spec A\to SL_n$?

Since the question is about \emph{all} $k$-algebras $A$ and all unimodular rows of length $n$ on $A$, it is reasonable to try to deal with the above question by looking at the universal algebra parameterizing unimodular rows of length $n$ (together with the choice of a section), namely the coordinate $k$-algebra $A_{2n-1}$ of $Q_{2n-1}$.  Indeed, let $A$ be a $k$-algebra and $v$ be a unimodular row of length $n$. Then the choice of $w\in Um_n(A)$ such that $v\cdot w^t=1$ yields a lift of the morphism $v:\Spec A\to \A^n\setminus 0$ to a morphism $v^\prime:\Spec A\to Q_{2n-1}$, i.e. we have a commutative diagram
\[
\xymatrix@C=3em{\Spec R\ar@{-->}[r]^-{v^\prime}\ar[rd]_-v & Q_{2n-1}\ar[d]^-{p_{2n-1}}\\
 & \A^n\setminus 0.}
\]
Let now $\varphi:\A^n\setminus 0\to \A^n\setminus 0$ be a morphism and $r:SL_n\to \A^n\setminus 0$ be the projection to the first row. The diagram
\[
\xymatrix{ & Q_{2n-1}\ar[d]^-{p_{2n-1}} & \\
\Spec R\ar[r]_-v\ar[ru]^-{v^\prime}\ar[rd]_{\varphi v} & \A^n\setminus 0\ar[d]^-\varphi & SL_n\ar[ld]^-r \\
 & \A^n\setminus 0 & }
\]
thus proves that it suffices to show that $\varphi\circ p_{2n-1}$ factorizes through $SL_n$ to prove that $\varphi v$ also factorizes through $SL_n$.

\begin{thm}
\label{thm:Nori}
If $n$ is an even integer, then \textup{Question \ref{quest:nori}} has an affirmative answer.
\end{thm}

\begin{proof}
The morphism $p_{2n-1}:Q_{2n-1}\to \A^n\setminus 0$ corresponds to the unimodular row $(x_1,\ldots,x_n)$, whose class in $H^{n-1}(Q_{2n-1},BSL_{n-1}) \cong H^{n-1}(Q_{2n-1},{\mathbf S}_n) = \Z/(n-1)!$ is $1$.

The unimodular row $(f_1,\ldots,f_n)$ corresponds to the morphism
\[
F:\A^n\setminus 0\to \A^n\setminus 0
\]
defined by $F(x_1,\ldots,x_n)=(f_1(x_1,\ldots,x_n),\ldots,f_n(x_1,\ldots,x_n))$. This morphism induces a commutative diagram
\[
\xymatrix{H^{n-1}(\A^n\setminus 0,\K^M_n)\ar[r]^-\partial\ar[d]_-{F^*} & H^n_{\{0\}}(\A^n,\K^M_n)\ar[d]^-\alpha \\
H^{n-1}(\A^n\setminus 0,\K^M_n)\ar[r]_\partial & H^n_{\{0\}}(\A^n,\K^M_n)  }
\]
where $\partial$ is the connecting homomorphism in the long exact sequence in cohomology associated with the open embedding $\A^n\setminus 0\subset \A^n$ and $\alpha$ is the map we want to identify. Consider the closed subset $Y$ of $\A^n\setminus 0$ given by the equations $x_1=\ldots=x_{n-1}=0$.  View the element $x_n$ as lying in $K_1^M(k(Y))$. A simple computation shows that $\partial (x_n)$ is the generator of $H^n_{\{0\}}(\A^n,\K^M_n)=\Z$.

Up to homotopy, we can suppose that the sequence $(f_1,\ldots,f_{n-1})$ is regular and even smooth by Swan's version of Bertini's theorem \cite[Theorem 1.3]{Swan74}. The closed subscheme $Z$ defined by $f_1=\ldots=f_{n-1}=0$ is then of height $n-1$ and $f_n$ can be seen as an element of $K_1^M(k(Z))$. We have $F^*(x_n)=f_n$ and by definition of the connecting homomorphism $\partial (f_n)=l(R/\langle f_1,\ldots,f_n\rangle)$ under the isomorphism $H^n_{\{0\}}(\A^n,\K^M_n)=\Z$.
\end{proof}

When $n$ is odd, the answer to Nori's question is known to be negative by \cite[Theorem 4.7]{Fasel11c}. In view of this counter-example, the second author proposed a refined version of Nori's question, which we now explain.  Lemma \ref{lem:cohomology} shows that $H^{n-1}(Q_{2n-1},{\bf I}^n) \cong H^{n-1}(\A^n\setminus 0,{\bf I}^n)\simeq ({\bf I}^n)_{-n}(k)=W(k)$.  The isomorphism can be uniquely specified by choosing a trivialization of the normal sheaf of $0$ in $\A^n$, i.e., an orientation class \cite[Remark 2.5]{Fasel11c}. Any morphism $\varphi:\A^n\setminus 0\to \A^n\setminus 0$ yields a homomorphism $\varphi^*:W(k)\to W(k)$ that we call the degree of $\varphi$ and write $\mathrm{deg}(\varphi)$. This degree is simply a concrete avatar of (the quadratic part of) F. Morel's Brouwer degree \cite[Corollary 24]{MField}; see also Corollary \ref{cor:GWdegree}.

\begin{rem}
In \cite{Fasel11c}, a degree homomorphism is defined by considering the Grothendieck-Witt group $GW_{red}^{n-1}(\A^n\setminus 0)$ (here the subscript {\em red} means ``reduced," i.e., one has split off the summand corresponding to a base-point; see \cite[Lemma 2.4]{Fasel11c} for more details). This degree is exactly the same as the one defined above. Indeed, the Gersten-Grothendieck-Witt spectral sequence $E(n-1)^{p,q}$ defined in \cite{FaselSrinivas} shows that the edge homomorphism $E(n-1)_2^{n-1,0}=H^{n-1}(\A^n\setminus 0,{\bf I}^n)\to GW^{n-1}(\A^n\setminus 0)$ induces an isomorphism $H^{n-1}(\A^n\setminus 0,{\bf I}^n)\to GW_{red}^{n-1}(\A^n\setminus 0)$.
\end{rem}

We now state and prove a result that constitutes a positive answer to a refinement of Nori's original question; this provides an answer to \cite[Question 4.8]{Fasel11c}.

\begin{thm}\label{thm:Nori+}
Suppose $k$ is a field, $R=k[x_1,\ldots,x_n]$ is a polynomial ring in $n$ variables over $k$, $\phi:R\to A$ is a $k$-algebra homomorphism such that $\sum \phi(x_i)A=A$, and $f_1,\ldots,f_n$ are as in Question \ref{quest:nori}.
Let $\varphi:\A^n\setminus 0\to \A^n\setminus 0$ be the morphism induced by $f_1,\ldots,f_n$. If the length of $R/(f_1,\ldots,f_n)$ is divisible by $(n-1)!$ and the degree $\mathrm{deg}(\varphi)=0$, then the unimodular row $(\phi(f_1),\ldots,\phi(f_n))$ is completable.
\end{thm}

\begin{proof}
Theorem \ref{thm:bundlesodd} shows that $Um_n(Q_{2n-1})/E_n(Q_{2n-1})$ is the fiber product of the groups $W(k)$ and $\Z/(n-1)!$ over $\Z/2$. The same arguments as in the proof of Theorem \ref{thm:Nori} show that we have to prove that the unimodular row $(f_1,\ldots,f_n)$ is completable if $\mathrm{deg}(\varphi)=0$ (the degree is trivial if and only if its class in $W(k)/k^\times$ is the trivial class) and $(n-1)!$ divides the length $l$ of $R/(f_1,\ldots,f_n)$. However, the unimodular row $(f_1,\ldots,f_n)$ corresponds to the pair $(\mathrm{deg}(\varphi),l)$ in the fiber product by definition of $\mathrm{deg}(\varphi)$ and Theorem \ref{thm:Nori}.
\end{proof}


\subsection*{Compatibility with realization}
Assume $k = \cplx$.  If $(X,x)$ is a pointed smooth scheme, and $X(\cplx)$ is the associated topological space of complex points, the assignment $X \mapsto X(\cplx)$ can be extended to a functor
\[
{\Re}_{\cplx}: \hop{\cplx} \longrightarrow {\mathcal H}_{\bullet},
\]
where ${\mathcal H}_{\bullet}$ is the usual topological homotopy category (see \cite[p. 120-121]{MV} for more details).  In particular if $({\mathscr X},x)$ is a pointed space, there are induced homomorphisms
\[
\bigoplus_{i+j = n} \bpi_{i,j}^{\aone}(\mathscr{X})(\cplx) \longrightarrow \pi_{n}(\mathscr{X}(\cplx))
\]
by summing the various component homomorphisms.

Taking $\mathscr{X} = SL_n$, complex realization allows us to compare the computations of $\aone$-homotopy sheaves of Theorem \ref{thmintro:mainhomotopysheaf} with those coming from classical homotopy theory.  We will see that the above homomorphism is surjective in some situations.  The precise description of first non-stable $\aone$-homotopy sheaves of $SL_n$ was motivated by anticipation of results such as those established here.

\subsubsection*{Compatibility with complex realization}
Bott periodicity (see also \cite[Theorem 5]{Bott}) yields a computation of the homotopy groups of the unitary group in the stable range and the first non-stable homotopy group:
\[
\pi_i(U(n)) = \begin{cases} 0 & \text{ if } i < 2n, i \text{ even } \\
\Z & \text{ if } i < 2n, i \text{ odd, and } \\
\Z/n! & \text{ if } i = 2n. \end{cases}
\]
Furthermore, it is classically known that $\pi_{5}(U(2)) = \Z/2$ (use \cite{Whiteheadpinplus2} together with the fact that $U(2)$ is an $S^1$-bundle over $SU(2)$) and $\pi_6(U(2)) = \Z/12$ \cite[Proposition 19.4]{BorelSerre}.

\begin{thm}
\label{thm:complexrealization}
For any integer $n \geq 3$, the homomorphisms
\[
\begin{split}
\bpi_{n-1,n}^{\aone}(GL_n)(\cplx) &\longrightarrow \pi_{2n-1}(GL_n(\cplx)) \cong \pi_{2n-1}(U(n)) = \Z, \text{ and} \\
\bpi_{n-1,n+1}^{\aone}(GL_n)(\cplx) &\longrightarrow \pi_{2n}(GL_n(\cplx)) \cong \pi_{2n}(U(n)) = \Z/n!,
\end{split}
\]
induced by complex realization are isomorphisms.
\end{thm}

\begin{proof}
For the first isomorphism, we proceed as follows.  The stabilization morphism
\[
\bpi_{n-1,n}^{\aone}(GL_n)(\cplx) \longrightarrow \bpi_{n-1,n}^{\aone}(GL_{n+1})(\cplx)
\]
is an isomorphism by Proposition \ref{prop:tatehomotopysheavesofgln}, and the latter group is isomorphic to the stable group $\bpi_{n-1,n}^{\aone}(GL)(\cplx)$.  In particular, as we observed before, both groups in question are isomorphic to $\Z$.  Since complex realization is functorial, we have a commutative diagram of the form
\[
\xymatrix{
\bpi_{n-1,n}^{\aone}(GL_n)(\cplx) \ar[r]\ar[d] & \bpi_{n-1,n}^{\aone}(GL)(\cplx) \ar[d] \\
\pi_{2n-1}(U(n)) \ar[r]& \pi_{2n-1}(U(\infty)).
}
\]
The bottom horizontal arrow is an isomorphism.

The right vertical morphism is also an isomorphism.  Indeed, we have a canonical isomorphism ${\mathbf R}\Omega^1_s(\Z \times BGL) \cong GL$, and adjunction therefore gives an identification
\[
\bpi_{n-1,n}^{\aone}(GL) \cong \bpi_{n,n}^{\aone}(\Z \times BGL_{\infty}).
\]
By Bott periodicity in the algebraic setting, i.e., the computation of \cite[\S 4, Theorem 3.13]{MV}, we observe that
\[
[{\pone}^{\wedge n},\Z \times BGL_{\infty}]_{\aone} \cong [S^0_\cplx,\Z \times BGL_{\infty}]_{\aone}.
\]
The latter group is isomorphic to $K_0^Q(\cplx)$, which is generated by the class of a $1$-dimensional $\cplx$-vector space.  Analogously, one identifies $\pi_{2n-1}(U(n)) \isomt \pi_{2n-1}(U(\infty)) \isomt \pi_{2n}(\Z \times BU(\infty))$.  Again, by Bott periodicity, the latter can be identified with the topological K-theory $K^0_{top}(pt)$, again generated by a $1$-dimensional complex vector space.  Keeping track of the various identifications, we see that a generator of $\bpi_{n-1,n}^{\aone}(GL_n)(\cplx)$ is mapped to a generator of $\pi_{2n-1}(U(n))$.

For the second isomorphism, we proceed as follows.  For any integer $n \geq 3$, we first identify $\bpi_{n-1,n+1}^{\aone}(GL_n) \cong \bpi_{n-1,n+1}^{\aone}(SL_n) \cong \bpi_{n,n+1}^{\aone}(BSL_n)$.  Because $\bpi_1^{\aone}(BSL_n)$ is trivial, the latter set can be canonically identified with the set of free $\aone$-homotopy classes of maps $[Q_{2n+1},BSL_n]_{\aone}$. Similarly, we identify $\pi_{2n}(GL_n(\cplx)) = \pi_{2n}(SU(n))$ with $[S^{2n+1},BSU(n)]$ by means of the clutching construction.

Since $W(\cplx) = \Z/2$, Theorems \ref{thm:bundleseven} and \ref{thm:bundlesodd} tell us that the set of isomorphism classes of oriented rank $n$ bundles on $Q_{2n+1}$ has a natural group structure and is isomorphic to $\Z/n!\Z$ (irrespective of whether $n$ is even or odd).  Now, the map that sends a complex algebraic vector bundle to the underlying topological vector bundle defines a function
\[
[Q_{2n+1},BSL_{2n}]_{\aone} \longrightarrow [S^{2n+1},BSU(n)].
\]
As mentioned above, topological vector bundles can be described by means of the clutching construction.  Now, each of the vector bundles of rank $n$ on $Q_{2n+1}$ is given by a unimodular row.  The homotopy class of the clutching function attached to the unimodular row is computed, e.g., in \cite[Theorem 3.1]{SwanTowber} and this gives the required isomorphism.
\end{proof}

\begin{rem}
Consider the homomorphism $\bpi_{n-1-i,j}^{\aone}(SL_n) \to \pi_{n-1-i+j}(SU(n))$.  If $i > 0$, then $\bpi_{n-1-i}^{\aone}(SL_n)$ is in the stable range and therefore isomorphic to $\K^Q_{n}$.  It follows that if $i > 0$ and $j > n$, then $\bpi_{n-1-i,j}^{\aone}(SL_n)$ is trivial, so the homomorphism from the first line is trivial.  If $i < 0$, we do not know what happens.
\end{rem}

\subsubsection*{Comments on real realization}
The map sending a $\real$-scheme $X$ the topological space $X(\real)$ with its usual topology can be extended to a functor $\ho{\real} \to {\mathcal H}$ that we refer to as ``real realization" (this is easier than the discussion of \cite[p. 121-122]{MV}, and analogous to the discussion of ``complex realization" presented there).


There is a homotopy equivalence $GL_n(\real) \cong O(n)$.  The group $O(2)$ is an extension of $\Z/2$ by $SO(2)$, and so the homotopy groups of $SO(2)$ are trivial in degree $> 1$.  The groups $\pi_{n-1}(O(n))$ for $n \geq 3$ are determined by Bott periodicity.  For completeness, we quote the result from \cite{Kervaire}: the group $\pi_{r-1}(O(r))$ is equal to $0, \Z \oplus \Z, \Z/2,\Z,0$ if $r = 3,4,5,6$ or $7$ and, more generally, $\Z \oplus \Z, \Z/2 \oplus \Z/2, \Z \oplus \Z/2, \Z/2, \Z \oplus \Z, \Z/2, \Z, \Z/2$ if $r \geq 8$, and $r \equiv 0,1,2,3,4,5,6$ or $7$ modulo $8$.  The situation involving compatibility with real realization is more subtle than that of complex realization.

Real realization gives rise to canonical homomorphisms
\[
\bpi_{i,j}^{\aone}(GL_n)(\real) \longrightarrow \pi_i(GL_n(R)) \cong \pi_i(O(n));
\]
in particular, $\aone$-homotopy groups of several different weights map to the {\em same} topological homotopy group.  If $n \geq 2$ and $i \geq 2$, we can again use fiber sequences to study $SL_n$ and $SO(n)$ instead of $GL_n$ and $O(n)$.  In that situation, the isomorphisms in question are compatible with the clutching construction (as above).

Similar to the situation involving complex realization, real realization is compatible with (simplicial) suspension, so the homomorphism above can also be identified as a morphism
\[
\bpi_{i+1,j}^{\aone}(BGL_n))(\real) \longrightarrow \pi_{i+1}(BO(n)).
\]
The computations of homotopy groups of $O(n)$ give rise to descriptions of the set of isomorphism classes of rank $n$ topological vector bundles on $S^{n}$.  Likewise, Theorems \ref{thm:bundlesodd} and \ref{thm:bundleseven} give descriptions of the sets of isomorphism classes of real rank $n$ vector bundles on $Q_{2n+1}$ (which has real realization homotopy equivalent to $S^n$): these groups are equal to $\Z/(n-1)!\Z$ if $n$ is odd and $\Z/(n-1)! \times_{\Z/2} W(k)$ if $n$ is even (the indices have shifted).

While neither realization map is (individually) surjective or injective, it is possible that the map $\bigoplus_{j}\bpi_{n,j}^{\aone}(BGL_n) \to \pi_{n}(BO(n))$ is surjective.  Nevertheless, the factor of $\Z$ that corresponds to $W(\real)$ in Theorem \ref{thm:bundlesodd} does admit an elementary explanation; we view the following remark as an explanation of the factors of ${\mathbf I}^n$ that appear in Theorem \ref{thm:homotopysheafevencase}.

\begin{rem}
A rank $i$ vector bundle on $S^n$ is classified by a map $S^n \to BSO(i)$.  The obvious inclusion $SO(i) \hookrightarrow SO(i+1)$ induces a map $BSO(i) \to BSO(i+1)$.  Those maps $S^n \to BSO(i)$ such that the composed maps $S^n \to BSO(i+1)$ are homotopically trivial (i.e., those rank $i$ vector vector bundles that become trivial upon direct sum with a trivial line bundle) lift to a map $\tilde{f}: S^n \to SO(i+1)/SO(i) \cong S^i$.  Taking $i = n$, the homotopy class of the map $\tilde{f}$ is completely determined by its topological degree.

Now, given a rank $n-1$ vector bundle on $Q_{2n-1}$ corresponding to a unimodular row, the classifying map $Q_{2n-1} \to BSL_{n-1}$ lifts to a map $Q_{2n-1} \to Q_{2n-1}$.  Morel has associated with such a map a degree in $GW(k)$, and there is an associated degree in $W(k)$; as observed in the proof of Theorem \ref{thm:Nori+}, this degree can be identified with the degree of \cite{Fasel11c}.  Taking $k = \real$, one observes that the real points of a map $Q_{2n-1} \to BSL_{n-1}$ correspond to a rank $n-1$ vector bundle on $S^{n-1}$ and the element of $W(\real)$ constructed above is precisely the topological degree of this map.
\end{rem}

\begin{footnotesize}
\bibliographystyle{alpha}
\bibliography{vectorbundlesonalgebraicspheres}

\begin{thebibliography}{CTHK97}

\bibitem[AD09]{ADExcision}
A.~Asok and B.~Doran.
\newblock {$\mathbb{A}^1$}-homotopy groups, excision, and solvable quotients.
\newblock {\em Adv. Math.}, 221(4):1144--1190, 2009.

\bibitem[AF12a]{Asok12}
A.~Asok and J.~Fasel.
\newblock A cohomological classification of vector bundles on smooth affine
  threefolds.
\newblock {\em Preprint}, available at \url{http://arxiv.org/abs/1204.0770},
  2012.

\bibitem[AF12b]{Asok12c}
A.~Asok and J.~Fasel.
\newblock Splitting vector bundles outside the stable range and homotopy theory
  of punctured affine spaces.
\newblock {\em Preprint}, available at \url{http://arxiv.org/abs/1209.5631},
  2012.

\bibitem[AF13]{AsokFaselEulerComparison}
A.~Asok and J.~Fasel.
\newblock Comparing {E}uler classes.
\newblock {\em Preprint}, available at \url{http://arxiv.org/abs/1306.5250},
  2013.

\bibitem[Aso13]{AsokPi1}
A.~Asok.
\newblock Splitting vector bundles and {${\mathbb A}^1$}-fundamental groups of
  higher dimensional varieties.
\newblock {\em J. Top.}, 2013.
\newblock doi:10.1112/jtopol/jts034.

\bibitem[Bot58]{Bott}
R.~Bott.
\newblock The space of loops on a {L}ie group.
\newblock {\em Michigan Math. J.}, 5:35--61, 1958.

\bibitem[BS53]{BorelSerre}
A.~Borel and J.-P. Serre.
\newblock Groupes de {L}ie et puissances r\'eduites de {S}teenrod.
\newblock {\em Amer. J. Math.}, 75:409--448, 1953.

\bibitem[CTHK97]{CTHK}
J.-L. Colliot-Th{\'e}l{\`e}ne, R.~T. Hoobler, and B.~Kahn.
\newblock The {B}loch-{O}gus-{G}abber theorem.
\newblock In {\em Algebraic {$K$}-theory ({T}oronto, {ON}, 1996)}, volume~16 of
  {\em Fields Inst. Commun.}, pages 31--94. Amer. Math. Soc., Providence, RI,
  1997.

\bibitem[DF11]{DuboulozFinston}
A.~Dubouloz and D.~R. Finston.
\newblock On exotic affine {$3$}-spheres.
\newblock {\em Preprint}, available at \url{http://arxiv.org/abs/1106.2900},
  2011.

\bibitem[Fas07]{Fasel07}
J.~Fasel.
\newblock The {C}how-{W}itt ring.
\newblock {\em Doc. Math.}, 12:275--312 (electronic), 2007.

\bibitem[Fas08]{FaselChowWitt}
J.~Fasel.
\newblock Groupes de {C}how-{W}itt.
\newblock {\em M\'em. Soc. Math. Fr. (N.S.)}, 113:viii+197, 2008.

\bibitem[Fas11a]{Fasel10}
J.~Fasel.
\newblock Projective modules over the real algebraic sphere of dimension 3.
\newblock {\em J. Algebra}, 325:18--33, 2011.

\bibitem[Fas11b]{Fasel08b}
J.~Fasel.
\newblock Some remarks on orbit sets of unimodular rows.
\newblock {\em Comment. Math. Helv.}, 86(1):13--39, 2011.

\bibitem[Fas12]{Fasel11c}
J.~Fasel.
\newblock A degree map on unimodular rows.
\newblock {\em J. Ramanujan Math. Soc.}, 27(1):23--42, 2012.

\bibitem[FS09]{FaselSrinivas}
J.~Fasel and V.~Srinivas.
\newblock Chow-{W}itt groups and {G}rothendieck-{W}itt groups of regular
  schemes.
\newblock {\em Adv. Math.}, 221(1):302--329, 2009.

\bibitem[Ker60]{Kervaire}
M.~A. Kervaire.
\newblock Some nonstable homotopy groups of {L}ie groups.
\newblock {\em Illinois J. Math.}, 4:161--169, 1960.

\bibitem[Kum97]{MohanKumarUnimodular}
N.~M. Kumar.
\newblock A note on unimodular rows.
\newblock {\em J. Algebra}, 191(1):228--234, 1997.

\bibitem[Mor04]{Morel04}
F.~Morel.
\newblock Sur les puissances de l'id\'eal fondamental de l'anneau de {W}itt.
\newblock {\em Comment. Math. Helv.}, 79(4):689--703, 2004.

\bibitem[Mor05]{MMilnor}
F.~Morel.
\newblock Milnor's conjecture on quadratic forms and mod 2 motivic complexes.
\newblock {\em Rend. Sem. Mat. Univ. Padova}, 114:63--101 (2006), 2005.

\bibitem[Mor12]{MField}
F.~Morel.
\newblock {\em {$\mathbb{A}^1$}-algebraic topology over a field}, volume 2052
  of {\em Lecture Notes in Mathematics}.
\newblock Springer, Heidelberg, 2012.

\bibitem[Mos11]{Moser}
L.-F. Moser.
\newblock {${\mathbb A}^1$}-locality results for linear algebraic groups
  [draft].
\newblock 2011.
\newblock {\em Preprint}.

\bibitem[MV99]{MV}
F.~Morel and V.~Voevodsky.
\newblock {$\mathbb{A}^1$}-homotopy theory of schemes.
\newblock {\em Inst. Hautes \'Etudes Sci. Publ. Math.}, 90:45--143 (2001),
  1999.

\bibitem[OVV07]{OVV}
D.~Orlov, A.~Vishik, and V.~Voevodsky.
\newblock An exact sequence for {$K\sp M\sb \ast/2$} with applications to
  quadratic forms.
\newblock {\em Ann. of Math. (2)}, 165(1):1--13, 2007.

\bibitem[Pan10]{Panin10}
I.~Panin.
\newblock Homotopy invariance of the sheaf {$W\sb {\rm Nis}$} and of its
  cohomology.
\newblock In {\em Quadratic forms, linear algebraic groups, and cohomology},
  volume~18 of {\em Dev. Math.}, pages 325--335. Springer, New York, 2010.

\bibitem[Ray68]{Raynaud}
M.~Raynaud.
\newblock Modules projectifs universels.
\newblock {\em Invent. Math.}, 6:1--26, 1968.

\bibitem[Ros96]{Rost96}
M.~Rost.
\newblock Chow groups with coefficients.
\newblock {\em Doc. Math.}, 1:No. 16, 319--393 (electronic), 1996.

\bibitem[Spa81]{Spanier}
E.~H. Spanier.
\newblock {\em Algebraic topology}.
\newblock Springer-Verlag, New York, 1981.
\newblock Corrected reprint.

\bibitem[ST75]{SwanTowber}
R.~G. Swan and J.~Towber.
\newblock A class of projective modules which are nearly free.
\newblock {\em J. Algebra}, 36(3):427--434, 1975.

\bibitem[Ste99]{Steenrod}
N.~Steenrod.
\newblock {\em The topology of fibre bundles}.
\newblock Princeton Landmarks in Mathematics. Princeton University Press,
  Princeton, NJ, 1999.
\newblock Reprint of the 1957 edition, Princeton Paperbacks.

\bibitem[Sus77]{SuslinStablyFree}
A.~A. Suslin.
\newblock Stably free modules.
\newblock {\em Mat. Sb. (N.S.)}, 102(144)(4):537--550, 632, 1977.

\bibitem[Sus82]{Suslin82b}
A.~A. Suslin.
\newblock Mennicke {S}ymbols and their application in the $k$-theory of fields.
\newblock In {\em Algebraic $K$-theory, Part I (Oberwolfach, 1980)}, volume 966
  of {\em Lecture Notes in Math.}, pages 334--356. Springer-Verlag, Berlin,
  1982.

\bibitem[Swa74]{Swan74}
Richard~G. Swan.
\newblock A cancellation theorem for projective modules in the metastable
  range.
\newblock {\em Invent. Math.}, 27:23--43, 1974.

\bibitem[Wen11]{Wendt}
M.~Wendt.
\newblock Rationally trivial torsors in {$\mathbb{A}^1$}-homotopy theory.
\newblock {\em J. K-Theory}, 7(3):541--572, 2011.

\bibitem[Whi50]{Whiteheadpinplus2}
G.~W. Whitehead.
\newblock The {$(n+2)^{\rm nd}$} homotopy group of the {$n$}-sphere.
\newblock {\em Ann. of Math. (2)}, 52:245--247, 1950.

\end{thebibliography}
\end{footnotesize}
\end{document}